\documentclass[12pt,reqno]{amsart}

\usepackage{amsmath,amsfonts,amsthm,mathrsfs,amssymb,amscd,enumerate,url,tikz-cd}

\input{xypic}
\xyoption{all}

\usepackage{xcolor}
\colorlet{mdtRed}{red!50!black}
\definecolor{dblue}{rgb}{0,0,.6}
\usepackage[colorlinks]{hyperref}
\hypersetup{linkcolor=blue,citecolor=dblue,filecolor=dullmagenta,urlcolor=mdtRed}

\newtheorem{theorem}{Theorem}[section]
\newtheorem{proposition}[theorem]{Proposition}
\newtheorem{lemma}[theorem]{Lemma}
\newtheorem{corollary}[theorem]{Corollary}
\numberwithin{equation}{theorem}

\theoremstyle{definition}
\newtheorem{definition}[theorem]{Definition}
\newtheorem{remark}[theorem]{Remark}
\newtheorem{rmk}[theorem]{Remark}

\newcommand{\Z}{\mathbb{Z}}

\newcommand{\Q}{\mathbb{Q}}

\newcommand{\X}{\mathfrak{X}}
\newcommand{\U}{\mathfrak{U}}
\newcommand{\w}{\omega}
\newcommand{\B}{\mathfrak{B}}
\newcommand{\E}{\mathcal{E}}
\newcommand{\bk}{{\bf k}}
\newcommand{\val}{\operatorname{val}}
\newcommand{\GL}{\operatorname{GL}}
\newcommand{\cone}{\operatorname{cone}}
\newcommand{\Spec}{\operatorname{Spec}}
\newcommand{\aff}{\textup{aff}}
\newcommand{\sph}{\textup{sph}}
\newcommand{\Loc}{\operatorname{Loc}}
\newcommand{\lin}{\operatorname{lin}}
\newcommand{\tail}{\operatorname{tail}}

\newcommand{\mf}[1]{\mathfrak{#1}}

\newcommand{\mb}[1]{\mathbb{#1}}
\newcommand{\mc}[1]{\mathcal{#1}}
\renewcommand{\t}[1]{\widetilde{#1}}
\newcommand{\D}{\mathfrak{D}}
\newcommand{\mS}{\mathcal{S}}
\newcommand{\rar}{\rightarrow}

\usepackage{marginnote}
\usetikzlibrary{calc}

\title[Equivariant Vector Bundles on complexity-one T-varieties]{Equivariant Vector Bundles on complexity-one T-varieties and Bruhat-Tits buildings}
\author[J. Dasgupta]{Jyoti Dasgupta}
\address{Department of Mathematics \& Computing, Indian Institute of Technology (Indian School of Mines) Dhanbad, Jharkhand, India. }
\email{jdasgupta.maths@gmail.com}

\author[C. Gangopadhyay]{Chandranandan Gangopadhyay}
\address{Department of Mathematics, Indian Institute of Science Education and Research, Pune, Maharashtra, India.}
\email{chandranandan@iiserpune.ac.in}

\author[K. Kaveh]{Kiumars Kaveh}
\address{Department of Mathematics, University of Pittsburgh, PA, USA}
\email{kaveh@pitt.edu}

\author[C. Manon]{Christopher Manon}
\address{Department of Mathematics, University of Kentucky, KY, USA}
\email{christopher.manon@uky.edu}
\date{today}

\begin{document}

\begin{abstract}
We give a combinatorial classification of torus equivariant vector bundles on a (normal) projective $T$-variety of complexity-one. This extends the classification of equivariant line bundles on complexity-one $T$-varieties by Petersen-S\"u\ss\ on one hand, and Klyachko's classification of equivariant vector bundles on toric varieties on the other hand. A main ingredient in our classification is the classification of torus equivariant vector bundles on toric schemes over a DVR in terms of piecewise affine maps to the (extended) Bruhat-Tits building of the general linear group.      
\end{abstract}

\maketitle

\tableofcontents

\section{Introduction}
We work over an algebraically closed field $\bk$ of characteristic $0$. In this paper, we give a combinatorial classification of torus equivaraint vector bundles on a (normal) projective $T$-variety of complexity-one. This extends the Klyachko classification of torus equivariant vector bundles on toric varieties on one hand (\cite{Klyachko}), and the Petersen-S\"u\ss\ classification of torus equivariant line bundles on complexity-one $T$-varieties on the other hand (\cite{PS}).

We recall that a $T$-variety is a (normal) variety $X$ together with an effective action of an algebraic torus $T$. The \emph{complexity} of a $T$-variety $X$ is $\dim X - \dim T$, in other words, the codimension of an orbit of maximal dimension. One knows that generic orbits have maximal dimension, and thus complexity is the dimension of ``quotient space'' of $X$ by $T$.
Toric varieties are (normal) $T$-varieties of complexity $0$.
In this paper, we are concerned with (normal) $T$-varieties of complexity $1$, that is, normal $T$-varieties where generic orbits form a $1$-dimensional family.

{\bf Toric varieties and equivariant vector bundles on them.} The geometry of toric varieties can be read off from combinatorics of convex polyhedral cones and convex polytopes. 
It is a classical fact that affine toric varieties are classified by polyhedral cones, and arbitrary toric varieties are classified by polyhedral fans. Torus equivariant line bundles on toric varieties are classified by integral piecewise linear functions with respect to the fan (\cite[Chapters 3, 6]{CLS}).

Let $X_\Sigma$ be a toric variety with fan $\Sigma$. Recall that a torus equivariant vector bundle (or a \emph{toric vector bundle} for short) on $X_\Sigma$ is a vector bundle $\E$ with a linear action of $T$ lifting that of $X_\Sigma$. In the remarkable work \cite{Klyachko}, extending the classification of toric line bundles in terms of integral piecewise linear functions (or support functions), Klyachko gives a classification of rank $r$ toric vector bundles in terms of compatible collections of (decreasing) $\Z$-filtrations in an $r$-dimensional vector space $E$. In \cite{KM}, the Klyachko data of a toric vector bundle $\E$ is interpreted as a \emph{piecewise linear map} from the fan $\Sigma$ to the (extended) Tits building of the group $\GL(E)$. Furthermore, in \cite{KMT}, this is extended to a classification of torus equivariant vector bundles on toric schemes over a discrete valuation ring $\mathcal{O}$ in terms of integral piecewise affine maps to the (extended) Bruhat-Tits building of $\GL(E)$. For a review of this material see Section \ref{sec-prelim}.   

{\bf $T$-varieties and equivariant line bundles on them.} On the other hand, extending the classification of toric varieties by polyhedral cones and fans, Altmann and Hausen (\cite{AH}) give a nice classification of $T$-varieties of arbitrary complexity by the data of polyhedral divisors (generalizations of polyhedral cones) and divisorial fans (generalizations of polyhedral fans). Building on the Altmann-Hausen classification, Petersen-S\"u\ss\ extend the classification of torus equivariant line bundles on toric varieties (in terms of piecewise linear functions on the fan) to $T$-varieties. Their classification is in terms of \emph{divisorial support functions} (\cite{PS}). We also mention \cite{IS-polarized} where, extending the classification of ample line bundles on toric varieties in terms of lattice polytopes, Ilten-S\"u\ss\ give a classification of ample line bundles on a complexity-one $T$-variety in terms of the data of \emph{divisorial polytopes}.

Let $T$ be a torus with character and cocharacter lattices $M$ and $N$, respectively. Let $Y$ be a smooth projective curve. Let $\Sigma_0$ be a fan in $N_\Q = N \otimes_\Z \Q$. We recall that a divisorial fan with tail fan $\Sigma_0$ is a certain collection $\mc{S} = (\mc S_P)_{P \in Y}$ of polyhedral complexes in $N_\Q$ with the same tail fan $\Sigma_0$. Moreover, for almost all $P \in Y$, the polyhedral complex $\mc S_P$ coincides with its tail fan $\Sigma_0$ (see Definition \ref{def-div-fan} and Definition \ref{def-slice}).
We let $X(\mc S)$ denote the complexity-one $T$-variety associated to the divisorial fan $\mc S$. In general, there is only a rational map from $X(\mc S)$ to $Y$. But one constructs a $T$-variety $\t{X}(\mc S)$ with a $T$-equivariant proper birational contraction $\t{X}(\mc S) \to X(\mc S)$ and together with a good quotient map $\pi: \t{X}(\mc S) \to Y$ (Section \ref{subsec-div-fan-T-var}).

We now explain the main results of the paper. Regarding notation, for the main part of the paper, we use notation closer to \cite{AH, AHS, PS}. In the background sections on toric schemes and buildings, we use notation closer to \cite{KM, KMT}.

{\bf Main result.}
Let $K = \bk(Y)$ be the field of rational functions of the curve $Y$. Let $X(\mc S)$ be the complexity-one $T$-variety corresponding to a divisorial fan $\mc S$. Let $E$ be an $r$-dimensional vector space over $K$. For each $P \in Y$, let $\val_P$ be the valuation on $K$ given by the order of vanishing at $P$. We let $\t{\B}_P(E)$ denote the \emph{extended Bruhat-Tits building} associated to the discretely valued field $(K, \val_P)$. We also denote the \emph{extended Tits building} of $E$ by $\t{\B}_\sph(E)$. Note that the extended Tits building does not require a choice of a valuation (see Section \ref{subsec-building}).

In \cite{KMT}, the notion of a \emph{piecewise affine map} from a polyhedral complex $\Sigma_1$ to an (extended) affine building $\t{\mf B}_\aff(E)$ is introduced. A piecewise affine map is a map that sends each polyhedron $\Delta$ in $\Sigma_1$ to an (extended) apartment $\t{A}_\aff(B_\Delta)$ in $\t{\mc B}_\aff(E)$ and $\Phi_{|\Delta}$ is given by an affine linear map (see Definition \ref{def-PA-map}). These are combinatorial gadgets classifying equivariant vector bundles on toric schemes over a DVR (see Section \ref{subsec-tvbs-toric-scheme} and  \cite[Sections 4.2 and 4.3]{KMT}).  

We define a \emph{support map} on a divisorial fan $\mc S$ to be a collection $(h_P: |\mc{S}_p| \to \t{\mc B}_P(E))_{P \in Y}$ of piecewise affine maps that satisfy certain natural compatibility conditions (see Definition \ref{support}).
The following is the main result of the paper (Theorem \ref{thm-main1}):
\begin{theorem}\label{thm-main1-intro}
There is an equivalence of categories between the category of $T$-equivariant vector bundles on $X(\mc S)$ and the category $\mc{C}(\mS)$ of support maps.
\end{theorem}

\begin{rmk}
For each point $P \in Y$, one has the toric scheme $\t{X}(\mc S)_P$ over $\Spec(\mc O_{Y, P})$ (see the diagram \eqref{equ-tild-X-P}). The map $h_P: |\mc S_P| \to \t{\B}_P(E)$ is the piecewise affine map associated to the pull-back of the equivariant vector bundle $\E$ to the toric scheme $\t{X}(\mc S)_P$ (see Theorem \ref{th:vb-dvr}).    
\end{rmk}

\begin{rmk}
(i) The piecewise linear map associated to the pull-back of $\E$ to the fiber over the generic point of $Y$, namely $\Spec(K)$, is given by the common linear part of the piecewise affine maps $h_P$ (see Theorem \ref{thm-main1}). 
(ii) Let $\pi: \t{X}(\mc S) \to Y$ be the quotient map. For almost all $P$, the fiber $\pi^{-1}(P)$ is a toric variety over $\bf k$ with the corresponding fan, the tail fan $\Sigma_0$ of $\mc S$. One can immediately recover the piecewise linear map associated to the restriction of $\E$ to the fiber $\pi^{-1}(P)$ from the piecewise affine map $h_P$ (see Remark \ref{rem-tvb-special-fiber}).   
\end{rmk}

In Sections \ref{sec-split} and \ref{sec-global-sec} we consider two applications of the main theorem (Theorem \ref{thm-main1-intro}):
\begin{itemize}
\item We say that a torus equivariant vector bundle is \emph{equivariantly split} if it is equivariantly isomorphic to a sum of equivariant line bundles. In Theorem \ref{th-splitting}, we give a criterion for equivariant splitting of $\E$ in terms of its support map $(h_P)_{P \in Y}$. This is an analogue of Klyachko's criterion for splitting of toric vector bundles (\cite{Klyachko}).

\item In Theorem \ref{thm-global}, we give a description of space of global sections $H^0(X(\mc S), \E)$ in terms of inequalities involving the support map of $\E$. This extends the corresponding statements in \cite{PS} and \cite{Klyachko} and is an extension of the well-known description of space of global sections of a toric line bundle on a toric variety in terms of lattice points in its Newton polytope. 
\end{itemize}

Finally, in Section \ref{sec-examples}, we describe the support maps of equivariant vector bundles in two classes of examples: (1) toric downgrades, that is, toric vector bundles over a toric variety regarded as a complexity-one $T$-variety for a codimension $1$ subtorus, and (2) (co)tangent bundles of complexity-one $T$-varieties. 

\begin{remark}
To extend our results to $T$-varieties of complexity $>1$, one needs to introduce the Bruhat-Tits buildings of general linear groups over valued fields with valuations of rank $>1$.      
\end{remark}

\begin{remark}
In the interesting work \cite{IS}, Ilten and S\"u\ss\ give a fully faithful functor from the category of equivariant vector bundles on a $T$-variety $X$ to a certain category of filtered vector bundles on a suitable quotient stack $Z$ of $X$ by $T$. When $X$ is factorial, they show that this functor gives an equivalence of categories. For an equivariant vector bundle on a complexity-one $T$-variety $X$, it is interesting to work out the relationship between the collection of filtered vector bundles on the quotient stack $Z$ (as in \cite{IS}) and the support map data (as in the present paper). In this regard, the well-known description of the affine Grassmannian of $\GL(E)$ over $\mc O_P$ (which can be identified with the set of lattice points in $\t{B}_P(E)$) as the moduli space of certain equivariant vector bundles over $\Spec(\mc O_P)$ should be relevant. 
\end{remark}

{\bf Further directions.} We end the introduction by proposing some problems for further research. As before, $\E$ denotes an equivariant vector bundle on a complexity-one $T$-variety $X$.
\begin{enumerate}
    \item[(i)] Give a description of equivariant Chern classes of $\E$ in terms of its associated support map $h=(h_P)_{P \in Y}$. The paper \cite{BKM}, which gives a description of equivariant Chern classes of the special fiber of a toric vector bundle on a toric scheme over a DVR, should be relevant to this.
    \item[(ii)] Give a criterion for (semi)stability of $\E$ in terms of its support map. This has been addressed in the toric case in the papers \cite{DDK, HNS, Devey}.
    \item[(iii)] Give criteria for ampleness and global generation of $\E$ in terms of its support map. In the toric case, such criteria have been given in \cite[Section 2]{HMP}, \cite[Theorem 1.2]{DJS} and \cite[Theorems 2.12 and 2.13]{KM}. The case of equivariant line bundles on complexity-one $T$-varieties has been addressed in \cite{IS-polarized}. 
\end{enumerate}

\bigskip
\noindent{\bf Acknowledgements.}
Christopher Manon was partially supported by Simons Collaboration Grant 587209 and National Science Foundation Grant DMS-2101911. Kiumars Kaveh was partially supported by a National Science Foundation Grant DMS-2101843 and a Simons Collaboration Grant. Chandranandan would like to thank Ankit Rai for helpful conversations. Finally, we thank Hendrik S\"u\ss\ and Nathan Ilten for useful comments.

\section{Background material}  \label{sec-prelim}
\subsection{Toric schemes over a DVR}
\label{subsec-toric-schemes}
Let $\mathcal{O}$ be a discrete valuation ring with fraction field $K$ and discrete valuation $\val: K \to \overline{\Z}:=\Z \cup 
\{\infty\}$. Let $\varpi$ be a uniformizer for $\mathcal{O}$, that is, a generator of the principal maximal ideal $\mathfrak{m}$ of $\mathcal{O}$, and we let $\bf k = \mathcal{O}/\mathfrak{m}$ be the residue field. We denote by $s$ and $\eta$ the special and generic points of $S=\Spec(\mathcal{O})$, corresponding to the maximal ideal $\mathfrak{m}$ and the prime ideal $(0)$, respectively. For a scheme $\X$ over $S$, we denote by $\X_s = \X \times_{S}\Spec(\bk)$ and by $\X_{\eta} = \X \times_S\Spec(K)$ the special and generic fiber, respectively.

We let $T$ be a split torus over $S$ of relative dimension $n$. We denote by $T_{\eta}$ and by $T_s$, the base change to $\operatorname{Spec}(K)$ and to $\operatorname{Spec}(\bk)$, respectively. We let $M$ be the character lattice and $N$ the cocharacter lattice of the torus $T_\eta \cong (K^*)^n$. We let $\widetilde{M} = M \times \Z$ and $\widetilde{N} = N \times \Z$. 

A \emph{toric scheme} $\X$ over $S$ of relative dimension $n$ is a normal integral separated scheme of finite type, equipped with a dense open embedding $T_{\eta} \hookrightarrow \mathcal{X}_{\eta}$ and an $S$-action of $T$ over $\mathcal{X}$ that extends the action of $T_{\eta}$ on itself by translations. The image of the identity point of $T_{\eta}$ in $\X$ is denoted by $x_0$. We note that if $\mathcal{X}$ is a toric scheme over $S$, then $\mathcal{X}_{\eta}$ is a toric variety over $K$ with torus $T_{\eta}$.

We briefly recall the classification of toric schemes over $S$ in terms of complete, strongly convex, rational polyhedral complexes (we refer to \cite[Section 3.5]{BPS} for details on the classification and to \cite[Section 2.1]{BPS} for definitions regarding strongly convex, rational polyhedral complexes and fans).

For the rest of the paper, by a cone, we always mean a polyhedral cone. Also, when we say cone/polyhedron/fan/polyhedral complex, we mean always a strongly convex, rational cone/polyhedron/fan/polyhedral complex. 

Let $\Sigma$ be a fan in $N_{\Q} \times \Q_{\geq 0}$. We denote by $\Sigma_1$ the polyhedral complex in $N_{\Q} \times \{1\}$ obtained by intersecting all the cones in $\Sigma$ by $N_{\Q} \times \{1\}$. 
Similarly, we denote by $\Sigma_0$ the fan in $N_{\Q} \times \{0\}$ obtained by intersecting all the cones in $\Sigma$ by $N_{\Q} \times \{0\}$.

\begin{remark}\label{polyhedra and cone}
    When $\Sigma$ is complete in $N_{\Q} \times \Q_{\geq 0}$, the fan $\Sigma_0$ coincides with the \emph{tail fan} (also called the \emph{recession fan}) of the polyhedral complex $\Sigma_1$ defined by:
\[
\tail(\Sigma_1) = \left\{\tail(\Delta) \times \{0\} \; | \; \Delta \in \Sigma_1\right\}.
\]
In this case, let $c(\Sigma_1)$ be the fan in $N_{\Q} \times \Q_{\geq 0}$ consisting of the cones over the polyhedra in $\Sigma_1$ together with their tail cones, i.e.
\[
c(\Sigma_1) := \left\{\cone(\Delta) \; | \; \Delta \in \Sigma_1\right\} \cup \tail(\Sigma_1).
\]
We note that $\Sigma_1 \mapsto \Sigma=c(\Sigma_1)$ gives a bijective correspondence between \emph{complete} polyhedral complexes in $N_{\Q}\times \{1\}$ and \emph{complete} fans in $N_{\Q} \times \Q_{\geq 0}$. If $\Sigma$ is not complete, this may not be true (see \cite{BS} for an example). 
\end{remark}

We now explain how to construct a toric scheme from a fan $\Sigma$ in $N_{\Q} \times \Q_{\geq 0}$. As in the case of toric varieties over a field, we do this by first associating an affine toric scheme $\X_{\sigma}$ to a cone $\sigma \in \Sigma$ and gluing them together. 

Let $\sigma \subset N_{\Q} \times \Q_{\geq 0}$ be a cone in $\Sigma$ with dual cone $\sigma^{\vee} \subset \t{M}_{\Q}$. Let $R_{\sigma}$ be the subring of the ring of Laurent polynomials $K[T_{\eta}]$ defined by 
\begin{equation}\label{affine-toric-scheme}
    R_{\sigma} = \mathcal{O}\left[\chi^u \varpi^k \; | \; (u,k) \in \sigma^{\vee}\cap \widetilde{M} \right],
\end{equation}
and let $\U_{\sigma} := \Spec(R_{\sigma})$. When $\sigma$ is the cone over a polyhedron $\Delta$ in $N_{\mb Q} \times \{1\}$, we also denote $R_{\sigma}$, $\U_\sigma$ by $R_{\Delta}$, $\U_\Delta$ respectively. 

As usual, one constructs the toric scheme $\X_{\Sigma}$ by gluing the affine schemes $\U_{\sigma}$ for $\sigma \in \Sigma$. If $\Sigma = c(\Sigma_1)$ for some complete polyhedral complex $\Sigma_1$ in $N_{\mb Q}\times \{1\}$, then we also denote the toric scheme $\X_{\Sigma}$ by $\X_{\Sigma_1}$. 

The next theorem contains some of the main properties of toric schemes (see \cite[Chap IV]{KKMD} and \cite[Section~3.5]{BPS} for proofs).
\begin{theorem}\label{th-toric-scheme-classification}
We have the following:
\begin{enumerate}
\item The association $\sigma \mapsto \U_{\sigma}$ gives an equivalence of categories between the category of cones in $N_{\mb Q}\times \mb Q_{\geq 0}$ and the category of affine toric schemes over $S$.
\item The association $\Sigma \mapsto \X_{\Sigma}$ gives an equivalence of categories between fans in $N_{\mb Q} \times \mb Q_{\geq 0}$ and toric schemes over $S$. Moreover, $\X_{\Sigma}$ is proper (respectively regular) if and only if $\Sigma$ is complete (respectively smooth). (Recall that $\Sigma$ is said to be smooth if any cone is generated by a subset of a basis of $N \times \Z$.)
\item We have two different types of cones in $\Sigma$: the ones that are contained in $N_{\mb Q} \times \{0\}$ and the ones that are not. The former are in bijective correspondence with the orbits of $T_{\eta}$ in $\X_{\Sigma,\eta}$. In particular, the generic fiber $\X_{\Sigma,\eta}$ is isomorphic to $X_{\Sigma_0}$, the toric variety over $K$ corresponding to the fan $\Sigma_0$. The cones that are not contained in $N_{\Q} \times \{0\}$ are of the form $\cone(\Delta)$ for some $\Delta \in \Sigma_1$ and are in bijective correspondence with the orbits of $T_{s}$ in the special fiber $\X_{\Sigma,s}$. Note that the special fiber $\X_{\Sigma,s}$ has an induced action by $T_{s}$, but, in general, it is not a toric variety over $\bk$. It may not be irreducible or reduced. It is reduced if and only if the vertices of all $\Delta \in \Sigma_1$ are in $N \times \{1\}$. In general, it is a union of toric varieties (with multiplicities) corresponding to the polyhedra in $\Sigma_1$, which intersect according to the combinatorics of $\Sigma_1$. This correspondence is order reversing. In particular, vertices in $\Sigma_1$ correspond to irreducible components of the special fiber, and two such components intersect if and only if there is a polyhedron in $\Sigma_1$ containing both vertices.  
\end{enumerate}
\end{theorem}

\subsection{Tits and Bruhat-Tits buildings of $\GL(r)$}  \label{subsec-building}
In this section we review some background material about the Tits and Bruhat-Tits buildings of the general linear group (following \cite{KMT}) as needed for the classification of toric vector bundles over a toric scheme.

As before, let $\val: K \to \overline{\Z}$ be a discretely valued field with valuation ring $\mc O$ and residue field $\bk$. Let $E$ be an $r$-dimensional vector space over $K$.

We consider two kinds of valuations on $E$: the ones that extend the trivial valuation on $K$, and the ones that extend $\val$. We incorporate both in the following definition.

\begin{definition}[Level $m$ additive norm]  \label{def-level-m-add-norm}
For $m \geq 0$, we call a function $\w: E \to \overline{\Q} = \Q \cup \{\infty\}$ a {\it level $m$ additive norm} if the following hold:
\begin{itemize}
	\item[(1)] For all $e \in E$ and $0 \neq \lambda \in K$ we have $\w(\lambda e) = m \val(\lambda) + \w(e)$. 
	\item[(2)] For all $e_1, e_2 \in E$, the non-Archimedean inequality $\w(e_1+e_2) \geq \min\{\w(e_1), \w(e_2)\}$ holds.
	\item[(3)] $\w(e) = \infty$ if and only if $e=0$. 
\end{itemize}
We refer to an additive norm of level $1$ simply as an \emph{additive norm},  and we call an additive norm of level $0$ a \emph{valuation}.
\end{definition}

\begin{definition}[Extended buildings]\label{def-ext-build}
The following objects play a central role in the paper:
\begin{itemize}
\item[(a)] We denote the set of all additive norms on $E$ by $\t{\B}_\aff(E)$ and call it the \emph{extended Bruhat--Tits building of $E$}. 
\item[(b)] We denote the set of all valuations on $E$ by $\t{\B}_\sph(E)$ and call it the \emph{extended Tits building of $E$} (or the \emph{cone over the Tits building of $E$}).
\item[(c)] For $m \geq 0$, we denote the set of all level $m$ additive norms by $\t{\B}_m(E)$. Clearly $\t{\B}_1(E) = \t{\B}_\aff(E)$ and $\t{\B}_0(E) = \t{\B}_\sph(E)$. We let $$\t{\B}(E) = \bigcup_{m \geq 0} \t{\B}_m(E),$$
and call it the \emph{total extended building} of $E$.
\end{itemize}
\end{definition}

\begin{remark}
We note that $\t{\B}(E)$ comes with a natural action of the multiplicative semigroup $\Q_{\geq 0}$: if $\w: E \to \overline{\Q}$ is a level $m$ additive norm and $k \geq 0$ then the function $k \w$ is a level $km$ additive norm. 
\end{remark}

A \emph{frame} is a direct sum decomposition of $E = \bigoplus_{i=1}^r L_i$ into one-dimensional subspaces $L_i$. In other words, frames correspond to vector space bases up to multiplication of basis elements by nonzero scalars. 

We say that a level $m$ additive norm $\w: E \to \overline{\Q}$ is \emph{adapted} to a frame $L=\{L_1, \ldots, L_r\}$ for $E$ if the following holds. For any $e = \sum_i e_i$, $e_i \in L_i$, we have: 
$$\w(e) =  \min\{ \w(e_i) \mid i=1, \ldots, r\}.$$
In other words, if $B = \{b_1, \ldots, b_r\}$ is a basis with $b_i \in L_i$, then for any $e = \sum_i \lambda_i b_i$ we have: $$\w(e) = \min\{ m\val(\lambda_i) + \w(b_i) \mid i=1, \ldots, r\}.$$ That is, the additive norm $\w$ is determined by its values on the basis $B$. 

\begin{definition}[Extended apartment] \label{def-ext-apt}
Let $L = \{L_1, \ldots, L_r\}$ be a frame for $E$. We denote the collection of all level $m$ additive norms that are adapted to $L$ by $\t{A}_m(L)$ and call it the extended level $m$ apartment of $L$. We put $\t{A}(L) = \bigcup_{m \geq 0} \t{A}_m(L)$. We also set 
$\t{A}_\sph(L) = \t{A}_0(L)$ and $\t{A}_\aff(L) = \t{A}_1(L)$.
If $B$ is a basis for $E$ spanning a frame $L$, we also denote the extended apartments corresponding to $L$ by $\t{A}(B)$, $\t{A}_\sph(B)$ and $\t{A}_\aff(B)$ respectively.
\end{definition}

\begin{remark}  \label{rem-ext-bldg-axioms}
 One shows that, for any $m>0$, any two level $m$ additive norms lie in the same extended apartment (this is, in fact, one of the defining axioms of an abstract building). In particular, $\t{\B}_m(E)$ is a union of extended apartments $\t{A}_m(L)$ (see \cite[Proposition 1.21]{RTW}).
\end{remark}

\begin{definition}
An $\mathcal{O}$-lattice in $E$ is a full rank $\mathcal{O}$-submodule $\Lambda \subset E$.
\end{definition}
The collection of all $\mathcal{O}$-lattices in $E$ is usually called the \emph{affine Grassmannian} of $\GL(E)$ and plays an important role in representation theory. To every $\mathcal{O}$-lattice $\Lambda$ there corresponds a $\Z$-valued additive norm $\w_\Lambda: E \to \overline{\Z}$ defined by $\w_\Lambda(e) = \max\{j \mid e \in \varpi^j \Lambda\}$. Conversely, if $\w$ is an integer-valued additive norm, $\Lambda_\w = E_{\w \geq 0} = \{e \in E \mid \w(e) \geq 0\}$ is an $\mathcal{O}$-lattice. One verifies that $\Lambda \mapsto \w_\Lambda$ and $\w \mapsto \Lambda_\w$ give a one-to-one correspondence between the set of $\mathcal{O}$-lattices in $E$ and the set of integer-valued additive norms. We think of integer-valued additive norms as lattice points in $\t{\B}_\aff(E)$.

\begin{definition}
Let $B = \{b_1, \ldots, b_r\}$ be a basis for $E$ spanning a frame $L$. One sees that for an $\mathcal{O}$-lattice $\Lambda$, the corresponding additive norm $\w_\Lambda$ is adapted to $L$ if and only if $$\Lambda = \sum_i \varpi^{a_i} b_i,$$ for some $a_i \in \Z$. We then say that $\Lambda$ is \emph{adapted} to the frame $L$ (or the basis $B$). 
\end{definition}

\begin{remark}[Geometric realizations of the Tits and Bruhat--Tits buildings]  \label{rem-realization-space-bldg}
One can give natural constructions of the geometric realization of the Tits building and Bruhat--Tits building of $\GL(E)$ from the spaces $\t{\B}_\sph(E)$ and $\t{\B}_\aff(E)$ as follows:
\begin{itemize}
    \item[(i)] Let us say that two valuations $\w$ and $\w'$ on $E$ are equivalent, written $\w \sim \w'$, if there exists $m>0$ and $c \in \Q$ such that $\w' = m\w + c$. Let $\mathcal{C}$ denote the set of all constant valuations, that is, valuations that have the same value on all nonzero vectors. Then one can show that the quotient space $\B_\sph(E) = (\t{\B}_\sph(E) \setminus  \mathcal{C}) / \sim$ can be identified with the geometric realization of the Tits building of $\GL(E)$. The apartments in $\B_\sph(E)$ are obtained from the corresponding quotients of extended apartments in $\t{\B}_\sph(E)$.
    \item[(ii)] Let us say that two additive norms $\w$ and $\w'$ on $E$ are equivalent, written $\w \sim \w'$, if there exists $c \in \Q$ such that $\w' = \w + c$. Then one can show that the quotient space $\B_\aff(E) = \t{\B}_\aff(E) / \sim$ can be identified with the geometric realization of the Bruhat--Tits building of $\GL(E)$. The apartments in $\B_\aff(E)$ are obtained from the quotients of extended apartments in $\t{\B}_\aff(E)$.
\end{itemize}
\end{remark}

As is well known in the theory of buildings, there are two ways one can associate spherical buildings to an affine building: (1) as the boundary at infinity of the affine building, and (2) as the link of any vertex.
The item (1) is related to the fact that the limit of $\t{\B}_m(E)$, as $m \to 0$, is $\t{\B}_\sph(E)$. The item (2) says that the \emph{link} of any vertex in an affine building (regarded as a simplicial complex) is a spherical building. The following proposition is a manifestation of this fact in the case of the Bruhat-Tits building of $\GL(E)$. For the proof, see \cite[Proposition 5.9]{BKM}. It is used in Theorem \ref{th-rest-irr-comp-sp-fiber}.

\begin{proposition} \label{prop-link-vertex}
Let $\Lambda$ be an $\mc O$-lattice in $E$ representing a vertex in the extended building $\t{\mf B}_\aff(E)$. Let $E_\Lambda = \Lambda \otimes_\mc O \bk$ be the corresponding $\bk$-vector space. 
\begin{itemize}
    \item[(a)]
A neighborhood $U_\Lambda$ of $\Lambda$ in the extended Bruhat--Tits building $\t{\B}_\aff(E)$ can naturally be identified with a neighborhood $U_\sph$ of the origin in the extended Tits building $\t{\B}_\sph(E_\Lambda)$.
\item[(b)]  There is a one-to-one correspondence between the extended apartments in $\t{\B}_\aff(E)$ that contain $\Lambda$ and the extended apartments in $\t{\B}_\sph(E_\Lambda)$, and under the identification of $U_\Lambda$ and $U_\sph$, apartments containing $\Lambda$ go to apartments. 
\end{itemize}
\end{proposition}

The above facts make appearances in Section \ref{sec-main} when we describe the fibers of an equivariant vector bundle $\E$ on a complexity-one $T$-variety $X=X(\mc S)$, over the generic point of $Y$ (i.e. over the field of fractions $K = \bk(Y)$) and over a $\bf k$-point $P$ of $Y$. 

\subsection{Equivariant vector bundles on toric schemes} \label{subsec-tvbs-toric-scheme}
We start by reviewing some of the main definitions needed for stating the classification of toric vector bundles over a DVR from \cite{KMT}. 

Let $\Sigma_0$ be a fan in $N_\Q = N_\Q \times \{0\}$.
\begin{definition}[Piecewise linear map to the extended Tits building]
We say that a map $\Phi_0: |\Sigma_0| \to \t{\B}_\sph(E)$ is a \emph{piecewise linear map} if the following hold:
\begin{itemize}
\item[(a)] For each cone $\sigma \in \Sigma_0$, there is a basis $B_\sigma \subset E$ (not necessarily unique) such that $\Phi_0(\sigma)$ lies in the extended apartment $\t{A}_\sph(B_\sigma)$.
\item[(b)] ${\Phi_0}|_{\sigma}$ is the restriction of a linear map from the linear span of $\sigma$ to $\t{A}_\sph(B_\sigma)$. \end{itemize}
We say that $\Phi_0$ is an \emph{integral} piecewise linear map if for each $\sigma \in \Sigma$, ${\Phi_0}|_{\sigma}$ is the restriction of an integral linear map from $N_\Q$ to $\t{A}_\sph(\sigma)$. Here, integral means it sends lattice points to lattice points.  
\end{definition}

\begin{definition}[Morphism of piecewise linear maps]  \label{def-morphism-PL}
Let $E$ and $E'$ be finite dimensional vector spaces. Let $\Phi_0: \Sigma_0 \to \t{\B}_\sph(E)$, $\Phi'_0: \Sigma_0 \to \t{\B}_\sph(E')$ be piecewise linear maps. A \emph{morphism} from $\Phi_0$ to $\Phi'_0$ is a linear map $F: E \to E'$ such that for all $x \in \Sigma_0$ and $e \in E$, we have $\Phi_0(x)(e) \leq \Phi'_0(x)(F(e))$.
\end{definition}

Next, let $\Sigma$ be a fan in $N_\Q \times \Q_{\geq 0}$. Recall that $\Sigma_1$ denotes the intersection of $\Sigma$ with $N_\Q \times \{1\}$. It is a polyhedral complex in $N_\Q \times \{1\}$.

\begin{definition}[Graded piecewise linear map to the total extended building]
We say that a map $\Phi: |\Sigma| \to \t{\B}(E)$ is a \emph{graded piecewise linear map} if the following hold:
\begin{itemize}
\item[(a)] For any $m \geq 0$ and $(x, m) \in |\Sigma| \subset N_\Q \times \Q$ we have $\Phi(x, m) \subset \t{\B}_m(E)$. That is, $\Phi$ sends level $m$ points to level $m$ points.
\item[(b)] For each cone $\sigma \in \Sigma$, there is a basis $B_\sigma \subset E$ (not necessarily unique) such that $\Phi(\sigma)$ lies in the extended apartment $\t{A}(B_\sigma)$.
\item[(c)] $\Phi|_{\sigma}$ is the restriction of a linear map from the linear span of $\sigma$ to $\t{A}(B_\sigma)$  
\end{itemize}
We say that $\Phi$ is an \emph{integral} graded piecewise linear map if for each $\sigma \in \Sigma$, $\Phi|_{\sigma}$ is the restriction of an integral linear map from $N_\Q \times \Q_{\geq 0}$ to $\t{A}(\sigma)$. 
\end{definition}

The notion of morphism of graded linear maps is defined in the same fashion as in Definition \ref{def-morphism-PL}.

Let $\Phi: |\Sigma| \to \t{\B}(E)$ be a piecewise linear map. Then 
the restriction $\Phi_1: |\Sigma_1| \to \t{\B}_\aff(E)$ of $\Phi$ to $|\Sigma_1| = |\Sigma| \cap (N_\Q \times \{1\})$ is a piecewise affine map in the following sense:
\begin{definition}[Piecewise affine map to $\t{\B}_\aff(E)$]\label{def-PA-map}
Let $\Sigma_1$ be a polyhedral complex in the affine space $N_\Q \times \{1\}$.
A map $\Phi_1: |\Sigma_1| \to \t{\B}_\aff(E)$ is a \emph{piecewise affine} map if the following holds:
\begin{itemize}
\item[(a)] For any polyhedron $\Delta \in \Sigma_1$, there is an extended apartment $\t{A}_\aff(\Delta)$ (not necessarily unique) such that the restriction ${\Phi_1}|_{\Delta}$ maps $\Delta$ into $\t{A}_\aff(\Delta)$.
\item[(b)] For any $\Delta \in \Sigma_1$, ${\Phi_1}|_{\Delta}$ is the restriction of an affine map from $N_\Q \times \{1\}$ to the affine space $\t{A}_\aff(\Delta)$.
\end{itemize}
We say that $\Phi$ is an \emph{integral} piecewise affine map if
for each $\Delta \in \Sigma_1$, ${\Phi_1}|_{\Delta}$ is the restriction of an integral affine map from $N_\Q \times \{1\}$ to $\t{A}_\aff(\Delta)$.
\end{definition}

Let $\Phi_1: |\Sigma_1| \to \t{\B}(E)$ be an integral piecewise affine map. For any polyhedron $\Delta \in \Sigma_1$, the requirement that $\Phi_{1, \Delta} := {\Phi_1}|_{\Delta}: \Delta \to \t{A}_\aff(\Delta)$ is an integral affine map means that there exists a $K$-basis $B_\Delta = \{b_{\Delta, 1}, \ldots, b_{\Delta, r}\}$ for $E$ and $u(\Delta) = \{u_{\Delta, 1}, \ldots, u_{\Delta, r}\} \subset M$ such that for any $x \in \sigma$, $\Phi_{1, \Delta}(x)$, as an additive norm on $E$, is given by:
\begin{equation}   \label{equ-Phi-sigma} 
\Phi_{1, \Delta}(x)\left(\sum_i \lambda_i b_{\Delta, i}\right) = \min\{ \val(\lambda_i) + \langle u_{\Delta, i}, x \rangle \mid i=1, \ldots, r\}.
\end{equation}

Finally, we recall the classification of toric vector bundles over a DVR from \cite{KMT}. 
\begin{theorem}[Classification of toric vector bundles over a DVR]  \label{th:vb-dvr}
To each toric vector bundle $\E$ over a toric scheme $\X_\Sigma$ there corresponds a graded piecewise linear map $\Phi_\E$ from $|\Sigma|$ to the total extended building $\t{\B}(E)$. Equivalently, when $\Sigma_0$ is the recession fan of $\Sigma_1$, to $\E$ we can associate the piecewise affine map $\Phi_{\E, 1} = {\Phi_{\E}}|_{|\Sigma_1|}: |\Sigma_1| \to \t{\B}_1(E) = \t{\B}_\aff(E)$. Moreover, this gives rise to an equivalence of categories. Furthermore, the restriction $\Phi_{\E, 0} = {\Phi_{\E}}|_{|\Sigma_0|} \colon |\Sigma_0| \to \t{\B}_0(E) = \t{\B}_{\sph}(E)$ is the piecewise linear map corresponding to the restriction of $\E$ to the generic fiber. 
\end{theorem}

A vertex $\nu$ of the polyhedral complex $\Sigma_1$ corresponds to an irreducible component $\X_{s, \nu}$ of the special fiber $\X_{\Sigma, s}$. If $\nu$ is an integral vertex, then this irreducible component is reduced and hence is a toric variety for the action of the residue torus $T_\bk = T \times_{\Spec(\mc O)} \Spec(\bf k)$. The fan of this toric variety is the star fan $\Sigma_\nu$ of the vertex $\nu$ in the polyhedral complex $\Sigma_1$. 

The piecewise linear map associated to the restriction of $\E$ to $\X_{s, \nu}$ can be described by means of the piecewise affine map $\Phi_{\E, 1}$ as follows (see \cite[Theorem 1.2]{BKM}). Suppose $\nu$ is an integral vertex of $\Sigma_1$ and let $\Lambda_\nu$ be the $\mc O$-lattice representing the lattice point $\Phi_1(\nu)$ in $\t{\B}_\aff(E)$. Also as in Proposition 
\ref{prop-link-vertex}, let 
$E_{\Lambda} = \Lambda_\nu \otimes_\mc O \bk$ be the corresponding $\bf k$-vector space.
\begin{theorem}[Restriction to an irreducible component of special fiber]  \label{th-rest-irr-comp-sp-fiber}
The piecewise linear map associated to the toric vector bundle $\E|_{\X_{s, \nu}}$ is the piecewise linear map $\Phi_\nu: |\Sigma_\nu| \to \t{\B}_\sph(E_\Lambda)$ obtained by restricting the piecewise affine map $\Phi_1$ to a sufficiently small neighborhood of the vertex $\nu$ (see Proposition \ref{prop-link-vertex}). 
\end{theorem}

\subsection{$T$-varieties} 
   By a T-variety, we mean a normal algebraic variety $X$ with an effective action of a torus $T$. The complexity of the $T$-variety $X$ is defined to be ${\rm dim}~X-{\rm dim}~T$. Toric varieties are $T$-varieties of complexity $0$. Altmann and Hausen extend the classification of toric varieties by polyhedral fans to $T$-varieties of arbitrary complexity. 
    In this subsection, we recall the construction of T-varieties from divisorial fans as done in \cite{AH} and \cite{AHS}. We will focus only on the complexity-one case.

\subsubsection{Proper polyhedral divisors and affine $T$-varieties}
Let $Y$ be a smooth curve (not necessarily projective) over the ground field $\bk$. We will begin by recalling the notion of a proper polyhedral divisor on $Y$ from \cite{AH}.

Let $M$ and $N$ be the character and cocharacter lattices of a torus $T$, respectively. We will denote the natural pairing between $M$ and $N$ by $\langle\cdot, \cdot \rangle$. For a polyhedron $\Delta\subset N_{\mathbb Q}:=N\otimes \mb Q$ we define its tail cone by 
$${\rm Tail}(\Delta):= \{ \, v\in N_{\mb Q} \, | \, v+\Delta\subset \Delta \, \}\,.$$
Let $\sigma$ be a pointed, convex, polyhedral cone in $N\otimes \mathbb Q$. Define
${\rm Pol}^+_{\sigma}(N_{\mathbb Q})$ as the set of all polyhedra in $N_{\mb Q}$ with tail cone $\sigma$. Under the Minkowski addition, the set ${\rm Pol}^+_{\sigma}(N_{\mathbb Q})$ is a commutative semigroup with cancellation and $\sigma$ as the identity element \cite[\S 1]{AH}. We denote the Grothendieck group of 
${\rm Pol}^+_{\sigma}(N_{\mathbb Q})$ by 
${\rm Pol}_{\sigma}(N_{\mathbb Q})$. Let 
$$\sigma^{\vee}=\{\, u \, \in  \, M_{\mb Q}\, |
\, \langle u,v \rangle \, \geq \, 0,~ \forall \, v \, \in \, \sigma \,\}$$ 
be the dual cone of $\sigma$.
Given $u\in \sigma$, we have the evaluation functional ${\rm eval}_u:{\rm Pol}_{\sigma}(N_{\mathbb Q}) \to \mb Q$ defined as follows (\cite[Proposition 1.7(iv)]{AH}). For $\Delta \, \in \, {\rm Pol}^+_{\sigma}(N_{\mathbb Q})$ we put:
\begin{equation}\label{eval}
{\rm eval}_u(\Delta)\, = 
\, {\rm min}\{\, \langle u,v \rangle \, | \, v \, \in \, \Delta \, \}\,.
\end{equation}
The group of \emph{polyhedral divisors} on $Y$ \cite[Definition 2.3]{AH} is defined as
$${\rm Div}_{\mathbb Q}(Y,\sigma):= \, {\rm Pol}_{\sigma}(N)\otimes_{\mb Z} {\rm Div}_{\mb Q}(Y)\,.$$
Here ${\rm Div}(Y)$ is the group of Weil $\mb Q$-divisors on $Y$. 
Any element of ${\rm Div}_{\mb Q}(Y, \sigma)$ is of the form 
$\sum_{P\, \in \, Y} \mf D_P\otimes P$ for 
$\mf D_P\in {\rm Pol}_{\sigma}(N_{\mathbb Q})$, where $\mf D_P=\sigma$ for all but finitely many $P$. For $u\in \sigma^{\vee}$, the evaluation functional ${\rm eval}_u$ above naturally extends to a map
\begin{align}\label{D(u)}
{\rm Div}_{\mb Q}(Y,\sigma) & \to {\rm Div}_{\mb Q}(Y) \\
\mathfrak D \, = \, \sum \mf D_P \otimes P & \mapsto \mathfrak D(u) \, := \sum {\rm eval}_u(\mf D_P)\, P  \,.
\end{align}
\begin{definition}[pp divisor]
    A polyhedral divisor $\mathfrak D$ on $Y$ is called a \emph{proper polyhedral divisor} (\emph{pp divisor} for short) if 
 it satisfies the following conditions:
 \begin{enumerate}
     \item If $\mathfrak D=\sum_{P\, \in \, Y} \mf D_P \otimes P$, where    
           $\mf D_P \in {\rm Pol}^+_{\sigma}(N_{\mathbb Q})$. This expression is unique by \cite[Proposition 1.7]{AH}.
     \item For any $u\in \sigma^{\vee}$ the Weil $\mb Q$-divisor
        $\mathfrak D(u)$ is $\mb Q$-Cartier, semiample and if 
        $u\in {\rm relint}(\sigma^{\vee})$ then $\mathfrak D(u)$ is also big.
 \end{enumerate}
 \end{definition}
 
    We refer to \cite[Example 2.12]{AH} for an equivalent formulation of the above definition.
    
 Next, we recall how to associate an affine $T$-variety to a pp divisor on $Y$. First, we fix the following notation:
Given a divisor $D\in {\rm Div}_{\mb Q}(Y)$, $D=\sum a_i \, D_i$ where $D_i$ are prime divisors, we denote by $\lfloor D \rfloor$ the (integral) Weil divisor $\sum \lfloor a_i \rfloor D_i$. Also, recall 
that we have the sheaf $\mathcal O(D)=\mathcal O(\lfloor D \rfloor)$ which on an open set $V\subset Y$ is given by 
\begin{equation}\label{O(D)}
H^0(V,\mc O(D)) \, = \, \{\, f\, \in \, \bk(Y)\, | \, {\rm div}(f)|_V + D|_V\, \geq 0\} \, = \, H^0(V,\mc O(\lfloor D \rfloor))\,.
\end{equation}
 Here $\bk(Y)$ is the function field of $Y$.

Given a polyhedral divisor $\mathfrak D$ on $Y$ one defines 
$$\mc A:= \bigoplus_{u\in \sigma^{\vee} \cap M} \mathcal{O}(\mathfrak D(u))\,.$$
The sheaf $\mc A$ is a sheaf of $\mc O_Y$-algebras \cite[Section  3]{AH}. Therefore, we can define
two schemes $\widetilde{X}(\mathfrak D)$ and $X(\mathfrak D)$ by setting
$$\widetilde{X}(\mathfrak D):={\underline{\rm Spec}}_Y\bigoplus_{u\in \sigma^{\vee} \cap M} \mathcal{O}(\mathfrak D(u)) \, ={\underline{\rm Spec}}_Y \, \mc A,$$
\begin{equation}\label{X(D)}
X(\mathfrak D):={\rm Spec}\bigoplus_{u\in \sigma^{\vee} \cap M} H^0 (\, Y \, , \, \mathcal{O}(\mathfrak D(u)) \,) \, = \, {\rm Spec}\, H^0(Y,\mc A)\,.
\end{equation}
Here, ${\underline{\rm Spec}}_Y$ denotes the relative spectrum over $Y$.
Then we have the following theorem:
\begin{theorem} \cite[Theorem 3.1]{AH} 
   \begin{enumerate}
     \item  
    The scheme $\widetilde X(\mf D)$ is a normal algebraic variety of dimension $1+ {\rm dim}~T$ and the grading of 
    $\mc A$ defines an effective action of $T$ on 
    $\widetilde X(\mf D)$ having the canonical map $\pi:\widetilde X(\mf D)\to Y$ as a good quotient.
    \item The ring $H^0(Y,\mc A)=H^0(X(\mf D),\mc O_{X(\mf D)})$ is a finitely generated $M$-graded normal $\bf k$-algebra and the natural map $\widetilde X(\mf D)\to X (\mf D) $ is a proper birational contraction.
   \end{enumerate}
 \end{theorem}  
  By \cite[Theorem 3.4]{AH}, 
    any affine $T$-variety of complexity-one is of the form $X(\mf D)$ for some 
    pp divisor $D\in {\rm Div}_{\mb Q}(Y,\sigma)$. To construct general $T$-varieties, one needs the notion of a divisorial fan, which we will recall next.

\subsubsection{Divisorial fans and $T$-varieties} \label{subsec-div-fan-T-var}
    To define divisorial fans, it is necessary to introduce the notion of polyhedral divisors where the coefficients are allowed to be empty \cite[\S 1]{AHS}. Formally, we do this by redefining polyhedral divisors on $Y$ to be elements of ${\rm Div}_{\mb Q}(U,\sigma)$ for any open subset $U$ of $Y$. Symbolically, if $\mf D\in {\rm Div}_{\mb Q}(U,\sigma)$, we will write 
    $\mf D = \sum_{P \in Y} \mf D_P \otimes P$ where $\mf D_P = \emptyset$ whenever $P \notin U$. From now on, we will use the term polyhedral divisor in the above sense.
    Given a polyhedral divisor $\mf D=\sum \mf D_P\otimes P$ on $Y$, we will denote the locus of all  $P\in Y$ such that $\mf D_P\neq \emptyset$ by ${\Loc}(\mf D)$. We define a polyhedral divisor $\mf D$ on $Y$ to be a pp divisor if it is a pp divisor when viewed as a polyhedral divisor on ${\Loc}(\mf D)$. We define the affine $T$-variety $X(\mf D)$ as in \eqref{X(D)}. 

    Suppose we are given two pp divisors $\mf D$ and $\mf D'$ on $Y$ such that $\mf D_P \subset \mf D'_P$ for all $P \in Y$. We call $\mf D$ a face of $\mf D'$ if the induced map $X(\mf D)\to X(\mf D')$ is an open embedding. See \cite[Proposition 3.4]{AHS} for an equivalent formulation of this definition. In particular, if $\sigma$ and $\sigma'$ are the tail cones of $\mf D$ and $\mf D'$ respectively, and if  $\mf D$ is a face of $\mf D$, then $\sigma$ is also a face of $\sigma'$.
\begin{definition}[Divisorial fan]\cite[Definition 5.2]{AHS}  \label{def-div-fan}
    A \emph{divisorial fan} on $(Y,N)$ is a finite set of pp divisors $\mc S$ such that for any two $\mf D,\, \mf D'\in \mc S$, 
    their intersection $\mf D \, \cap \, \mf D':= \sum (\mf D_P \cap \mf D'_P) \otimes P$ is a face of both $\mf D$ and $\mf D'$, and moreover, belongs to $\mc S$.
 \end{definition}
 For $\mf D,\, \mf D' \in \mc S$ define $X_{\mf D \mf D'}\subset X(\mf D)$ to be the image of the open embedding 
 $X(\mf D \, \cap \, \mf D')\to X(\mf D)$.
   \begin{theorem}\cite[Theorems 5.3 and 5.6]{AHS}
       The affine $T$-varieties $X(\mf D)$ and isomorphisms 
       $X_{\mf D\mf D'}\to X_{\mf D'\mf D}$ are gluing data. The resulting space $X(\mc S)$ is a $T$-variety of complexity-one. Conversely, every $T$-variety arises from a divisorial fan.
   \end{theorem}
   Here, we are also using the fact that $Y$ is a curve, in which case the variety $X(\mc S)$ is separated \cite[Remark 7.4]{AHS}.
 
\begin{remark}  \label{rem-tilde-X-S}
We note that if $\mc S$ is a divisorial fan on $Y$, the schemes $\widetilde{X}(\mf D)$, for $\mf D\in \mc S$, also glue together to form a $T$-variety $\widetilde{X}(\mc S)$. Furthermore, we have a diagram
\[
\begin{tikzcd}
\widetilde{X}(\mc S) \ar[r,"r"] \ar[d,"\pi"] & X(\mc S) \ar[d] \ar[dl, dashed] \\
Y \ar[r] & \Spec H^0(Y,\mc O_Y).
\end{tikzcd}
\]
where $r$ is a birational contraction and $\pi$ is  flat (see
\cite[Remark 1.10]{Petersen} and \cite[\S 1.7]{Vollmert}). We also recall that we are not assuming $Y$ to be projective in this section.
\end{remark}

We end this section by recalling a fact about open invariant covers of affine $T$ -varieties, which we will need later. Let $X(\mf D)$ be the affine $T$-variety associated with a pp divisor $\mf D=\sum \mf D_P\otimes P$ on $Y$.  Suppose we have an open affine cover $\mc U$ of $\Spec H^0(\Loc(\mf D),\mc O)$. For $U\in \mc U$, we define $$\mf D|_{q^{-1}(U)}:=\sum_{P\in U} \mf D_P \otimes P\,.$$
Here $q$ is the natural map ${\Loc(\mf D)}\to \Spec H^0(\Loc(\mf D),\mc O)$.
Then, the $T$-varieties $X(\mf D|_{q^{-1}U})$ along with the natural map $X(\mf D|_{q^{-1}U})\to X(\mf D)$ are open immersions, and they form an open cover of $X(\mf D)$ (see \cite[Proposition 3.3]{AHS}, \cite[Example 1.19]{Suss} and \cite[Lemma 2.6]{PS}). Moreover, we have the following:
\begin{lemma}\cite[Lemma 1.21]{Suss} \label{cover} 
    Let $X(\mf D)$ be the affine $T$-variety associated with a polyhedral divisor $\mf D$ on $Y$. Let $\mc V$ be an open invariant cover of $X(\mf D)$. Then, there exists a cover $\mc U$ of 
    $\Spec H^0(\Loc(\mf D),\mc O)$ 
    such that 
the cover $\{X(\mf D|_{q^{-1}(U)})\,|\, U\, \in \, \mc U\}$ of $X(\mf D)$ refines $\mc V$. In other words,
for any $U\in \mc U$, there exists $V\in \mc V$ such that $X(\mf D|_{q^{-1}(U)})\subset V$.
In particular, if $\Loc(\mf D)$ is complete, then $X(\mf D)\in \mc V$.
\end{lemma}

\subsubsection{Slices and toric schemes} Let us now assume that $Y$ is a  smooth projective curve over $\bf k$ and let us fix a divisorial fan $\mc S$ on $(Y,N)$. 
\begin{definition}[Slice] \label{def-slice}
Let $P$ be a closed point in $Y$.
    We define the slice of the divisorial fan $S$ at P by
$\mc S_P:=\{\, \mf D_P \,|\, \mf D \, \in \, \mc S \, \}$.
\end{definition}

For $P\in Y$, let us denote the local ring of $Y$ at $P$ by $O_{Y,P}$. Define $\widetilde X(S)_P$ to be the fiber product:
\begin{equation} \label{equ-tild-X-P}
\begin{tikzcd}  
    \widetilde X(\mc S)_P \ar[r] \ar[d] & \widetilde X(\mc S) \ar[d, "\pi"] \\
    \Spec \mc O_{Y,P} \ar[r]       & Y  
\end{tikzcd}               
\end{equation}

Note that since the map $\pi$ is $T$-invariant, the scheme $\widetilde X(S)_P$ is a toric scheme over $\Spec \mc O_{Y,P}$ with respect to the split torus $T\times \Spec \mc O_{Y,P}$.  

Recall that $c(\Sigma_1)$ denotes the \emph{cone} over a polyhedral complex $\Sigma_1$ in $N_\Q$. It is a fan in $N_\Q \times \Q_{\geq 0}$ obtained by taking cones over the polyhedra in $\Sigma_1$ (see Remark \ref{polyhedra and cone}). Also, we recall that to each fan in $N_\Q \times \Q_{\geq 0}$, there corresponds a toric scheme (Theorem \ref{th-toric-scheme-classification}). 

\begin{lemma} \label{lem-slice-toric-scheme}
The toric scheme $\widetilde X(\mc S)_P$ is the toric scheme associated to the fan $c( {\mc S}_P)\subset N_{\mb Q}\times \mb Q_{\geq 0}$.
\end{lemma} 
\begin{proof}
   Let $\mf D\in \mc S$. We show that $\widetilde X(\mf D)_P$ is  the toric scheme associated to $ \sigma:= {\rm cone}({\mf D}_P)$. 
   
   Let $v_i$ be the vertices of $\mf D_P$ and $\rho_j$ be ray generators of the tail cone $\tau:= {\rm tail}(\mf D_P)$. Then 
   $$\mf D_P= \{\sum a_iv_i + \sum b_j\rho_j \, | \, a_i, \, b_j\, \in \, \mb Q,\, 0 \leq a_i\leq 1, \,, 0\leq b_j \,, \sum a_i\, =\, 1\}.$$
  Now, we claim that 
  $$ \sigma = \mb Q_{\geq 0}(\{(v_i,1),\, (\rho_j,0)\})\,.$$
Indeed, by definition \cite[\S 2.1]{BPS} 
$$ \sigma := \, \overline{\mb Q_{\geq 0}(\mf D_P \times \{1\}})\,.$$

  As noted above, any element of $\mf D_P$ is of the form 
  $ \sum a_iv_i + \sum b_j\rho_j$ with $\sum a_i=1$ and $b_j\geq 0$. Therefore
  $$( \sum a_iv_i + \sum b_j\rho_j,1)=\sum_i a_i(v_i,1)+\sum_j b_j(\rho_j,0)\,.$$
  Therefore, the inclusion 
  $$ \sigma \subset \mb Q_{\geq 0}(\{(v_i,1),\, (\rho_j,0)\})\,.$$
  is clear. For the reverse inclusion, we only need to check that $(\rho_j,0)\in \sigma$. This can be seen as follows: fix $v\in \mf D_P$. Then $\forall~t>0$, $(v+t\rho_j,1)\in \sigma$. Therefore, $(v/t+\rho_j,1/t)\in \sigma$. Taking $t\rightarrow 0$, we get $(\rho_j,0)\in \sigma$. This proves the claim.

 Recall that the toric scheme associated with $ \sigma$ is given by 
 $$\mf U_{\sigma}= {\rm Spec} \, \mc O_{Y,P}\left[\chi^u \varpi_P^k \; | \; (u,k) \in \sigma^{\vee}\cap \widetilde{M} \right]\,$$ (see \eqref{affine-toric-scheme}).
 Here $\varpi_P$ is a uniformizer of $P$ in $\mc O_{Y,P}$.
  Since $\sigma$ is generated by the 
  $(v_i,1),\, (\rho_j,0)$, we get
  \begin{align*}
      (u,k)\in \sigma^{\vee}\cap \widetilde{M} \iff &  \, \langle (u,k),(v_i,1) \rangle, \, \langle(u,k), \, (\rho_j,0)\rangle \, \geq 0, ~~ \forall~i,j \geq 0\,, \\
      \iff & \langle u,v_i \rangle \geq -k,\, \langle u,\rho_j \rangle \geq 0,~~\forall~i,j \\
       \iff & k \geq - \lfloor {\rm eval}_u(\mf D_P) \rangle \rfloor, \, u\in \tau^{\vee}\cap M, 
  \end{align*}
  (see (\ref{eval}) for the definition of ${\rm eval}_u$). Therefore,
  $$\mf U_{\sigma}= {\rm Spec} \, \bigoplus\limits_{
  u\in \tau^{\vee}\cap M
  } \mc O_{Y,P}\varpi_{P}^{-\lfloor {\rm eval}_u(\mf D_P) \rfloor}\chi^{u}. $$
  
  Now, by definition of $\widetilde X(\mf D)$ (see (\ref{X(D)})), it follows that
  $$\widetilde X(\mf D)_P= {\rm Spec} \bigoplus\limits_{
  u\in \tau^{\vee}\cap M
  } \mc O(\mf D(u))_P \chi^u\,.$$
  Here $\mc O(\mf D(u))_P$ is the localisation of the sheaf $\mc O(\mf D(u))$ at the point $P$. The module $\mc O(\mf D(u))_P$ is same as $\mc O(\lfloor \mf D(u) \rfloor)_P$ by (\ref{O(D)}) and therefore, is generated as an $\mc O_{Y,P}$ module by the element $-\lfloor {\rm ord}_P(\mf D(u)) \rfloor = -\lfloor {\rm eval}_u(\mf D_P) \rfloor$ by (\ref{D(u)}). Hence  $\mf U_{\sigma}=\widetilde X(\mf D)_P$. This completes the proof of the lemma.
\end{proof}

\section{Main theorem}  \label{sec-main}
Let $T$ be an algebraic torus with character lattice $M$ and cocharacter lattice $N$. Let $Y$ be a smooth projective curve over $\bk$ and $\mc{S}$ be a divisorial fan on $Y$. Let $X(\mc S)$ be the associated $T$-variety. Let $K = \bk(Y)$ denote the function field of $Y$ and $E$ be an $r$-dimensional vector space over $K$. For each $P \in Y$, let us denote the extended Bruhat-Tits building of $E$ (with respect to the discrete valuation $\val_P$ on $K$) by $\t{\B}_P(E)$. Also let us denote the extended Tits building of $E$ by $\t{\B}_\sph(E)$ (see Definition \ref{def-ext-build}). Let $\eta$ be the generic point of $Y$. Then we have the toric variety $\t{X}(\mS_{\eta})$ with the action of the torus $T_{\eta}$. We identify $T_{\eta}$ with the open orbit in $\t{X}(\mS_{\eta})$ by fixing a point $x_0$ in the open orbit corresponding to the identity $1 \in T_{\eta}$.

Let $\E$ be a $T$-equivariant vector bundle of rank $r$ over a $T$-variety $X$. Then $\E$ is said to be \emph{equivariantly trivial} if there is a $T$-equivariant isomorphism between $\E$ and $X \times \bk^r \rar X$, where $T$ acts on $X \times \bk^r$ diagonally via a homomorphism $T \rar \GL(r, \bk)$.
\begin{lemma}\label{trivial}
Let $\mc E$ be an equivariantly trivial $T$-equivariant vector bundle over a $T$-variety $\t{X}(\D)$ over a smooth projective curve $Y$. For each $P \in Y$, we consider the pullback bundle $\mc{E}_P$ over the affine toric scheme $\t{X}(\D_P)$ (see Lemma \ref{lem-slice-toric-scheme}). Let $E$ be the fiber of the vector bundle $\mc{E}_P$ over $x_0$. Then for each $P \in \Loc(\D)$, we have an affine linear function $h_P:  \D_P \rar \t{\B}_P(E)$. Moreover, there is a basis $b_{\D,1}, \ldots, b_{\D,r}$ of $E$ and $u_{\D,1}, \ldots, u_{\D, r} \in M$ such that \[h_P(x)(\sum_{i=1}^r \lambda_i b_{\D,i})=  \min_i\{\val_P(\lambda_i)+\langle u_{\D, i}, x \rangle\}\] for all $P \in \Loc(\D)$ and for all $x \in \D_P$. Here $\lambda_i \in K$, for $i=1, \ldots, r$.
\end{lemma}
\begin{proof} By Theorem \ref{th:vb-dvr}, we have an affine linear function $h_P: \D_P \rar \t{\B}_P(E)$.
    Using the equivariant triviality of $\mc E$, we will have $M$-homogeneous sections $s_1, \ldots, s_r \in H^0(\t{X}(\D), \mc E)$ such that \[H^0(\t{X}(\D), \mc E) \cong \bigoplus_{i=1}^r H^0(\t{X}(\D), \mc O)s_i\] as $M$-graded $H^0(\t{X}(\D), \mc O)$-modules. As $\eta \in \Loc(\D)$, we have the inclusion
\begin{equation} \label{generic}
    T_{\eta} \hookrightarrow \t{X}(\D).
\end{equation}

Since $\mc {E} |_{T_{\eta}}$ is trivial, we have the identification $\mc {E} |_{T_{\eta}}\cong \mc {O}_{T_{\eta}} \otimes E$ and hence $H^0(T_{\eta}, \mc {E} |_{T_{\eta}})= H^0(T_{\eta},\mc {O}_{T_{\eta}} \otimes E) =K[T_{\eta}] \otimes E$. Thus, from \eqref{generic}, we have the inclusion 
\[H^0(\t{X}(\D), \mc{E})\hookrightarrow H^0(T_{\eta}, \mc {E} |_{T_{\eta}})=K[T_{\eta}] \otimes E.\] 
Let the image of $s_i$ under this inclusion be denoted by $\chi^{u_{\D,i}} \ b_{\D,i}$. Now, for $P \in \Loc(\D)$, we have \[H^0(\t{X}(\D_P), \mc E) \cong \bigoplus_{i=1}^r H^0(\t{X}(\D_P), \mc O)s_i|_{\t{X}(\D_P)}\] as $M$-graded $H^0(\t{X}(\D_P), \mc O)$-modules. Now, from the commutative diagram
 \[
\begin{tikzcd}
    \widetilde T{\eta} \ar[rr, hook] \ar[rd, hook] & & \widetilde X(\mf D) \\
    &   \t{X}(\D_P)\ar[ru, hook]    &   
\end{tikzcd}           
           \]
           we have that $s_i|_{\t{X}(\D_P)}$ maps to $\chi^{u_{\D,i}} \ b_{\D,i}$ under the inclusion \[H^0(\t{X}(\D_P), \mc{E})\hookrightarrow K[T_{\eta}] \otimes E.\] Then it follows from the definition of $h_P$ that \[h_P(x)(\sum_{i=1}^r \lambda_i b_{\D,i})=  \min_i\{\val_P(\lambda_i)+\langle u_{\D, i}, x \rangle\}\] for all $P \in \Loc(\D)$ and for all $x \in \D_P$. Here $\lambda_i \in K$ for $i=1, \ldots, r$.
\end{proof}
The above lemma motivates us to define the following:
\begin{definition}[Support map] \label{support}
     A \emph{support map} on a divisorial fan $\mc{S}$ is a  collection $h = (h_P: |\mc{S}_P| \rar \t{\B}_P(E))_{P \in Y}$, where $E$ is a finite dimensional $K$-vector space such that the following holds.
    \begin{enumerate}
        \item Each $h_P$ is a piecewise affine linear function that is affine linear on each polyhedron in $\mc{S}_P$ (see Definition \ref{def-PA-map}).
        \item $\lin(h_P)=\lin(h_P'): \tail(\mS) \rar \t{\B}_\sph(E)$ for all $P, P' \in Y$ (we note that for $P, P' \in Y$, we have $\tail(\mS_P)=\tail(\mS_{P'})=\tail(\mS)$).
        \item If $\D \in \mc S$ has affine locus, then there is a basis $b_{\D,1}, \ldots, b_{\D,r}$ of $E$ and $u_{\D, 1}, \ldots, u_{\D, r} \in M$ such that we have
        \[h_P(x)(\sum_{i=1}^r \lambda_i b_{\D,i})=  \min_i\{\val_P(\lambda_i)+\langle u_{\D, i}, x \rangle\},\] for all $\lambda_i \in K$, $i=1, \ldots, r$, all but finitely many $P \in Y$ and all $x \in \D_P$. 

        \item If $\D \in \mc S$ has complete locus, then there is a basis $b_{\D,1}, \ldots, b_{\D,r}$ of $E$ and $u_{\D, 1}, \ldots, u_{\D, r} \in M$ such that we have
\[h_P(x)(\sum_{i=1}^r \lambda_i b_{\D,i})=  \min_i\{\val_P(\lambda_i)+\langle u_{\D, i}, x \rangle\},\] for all $\lambda_i \in K, \ i=1, \ldots, r$, all $P \in Y$ and all $x \in \D_P$.
    \end{enumerate}  
    A \emph{morphism} between two support maps $h=(h_P: | \mc{S}_P| \rar \t{\B}_P(E))_{P \in Y}$ and $h'=(h'_P: | \mc{S}_P| \rar \t{\B}_P(E'))_{P \in Y}$ is defined to be a linear map of $K$-vector spaces $\phi: E \rar E'$ such that $\phi$ gives a morphism of piecewise affine maps $h_P \rar h'_P$ for all $P \in Y$ (see Section \ref{subsec-tvbs-toric-scheme}).

We denote by $\mc C (\mc{S})$ the category of support maps on the divisorial fan $\mS$.
\end{definition}
\begin{rmk}
    If $E$ is one-dimensional, choosing an isomorphism of $E$ with $K$ gives us an identification of $\t{\B}_P(K)$ with $\Q$, where each additive norm $v$ on $K$ is identified with 
    $v(1)$. Then condition (4) of Definition \ref{support} simplifies to the following. There is $f \in K$ and $u \in M$ such that for all $P \in Y$ and $\lambda \in K$ we have: \[h_P(x)(\lambda f)= \val_P(\lambda) + \langle u, x \rangle .\] In particular, \[h_P(x)(1)= -\val_P(f) + \langle u, x \rangle.\]Thus, we recover the data of a divisorial support function in \cite[Definition 3.4]{PS}.

Recall that in \cite[Definition 3.4]{PS}, it is required that in a divisorial support function $(h_P)_{P \in Y}$, except for finitely many $P$, $h_P$ coincides with its linear part. We would like to point out that when $\dim(E) > 1$, this condition does not directly make sense as the target space for $h_P$ is the extended Bruhat-Tits building $\t{\B}_P(E)$ while the target space for $\lin(h_P)$ is the extended Tits building $\t{\B}_\sph(E)$.
\end{rmk}

\begin{theorem}[Main theorem] \label{thm-main1}
We have the following:
\begin{enumerate}
\item There is an equivalence of categories between the category of $T$-equivariant vector bundles on $X(\mc S)$ and the category of support maps $\mc{C}(\mS)$. 
\item Let $\E$ be the $T$-equivariant vector bundle on $X(\mS)$ corresponding to a support map $h=(h_P)_{P \in Y}$. Then, the piecewise linear map associated to the toric vector bundle obtained by pulling back $\E$ to the fiber over the generic point of $Y$, i.e. the point over $K=\bk(Y)$, is given by the common linear part of the piecewise affine maps $h_P$.
\end{enumerate}       
\end{theorem}
    
\begin{proof}
Let $\mc E$ be a $T$-equivariant vector bundle on $X(\mS)$. We first consider the pullback of $\mc E$ to $\t{X}(\mS)$. Then, for each $P \in Y$, we further pull it back to the toric scheme $\t{X}(\mS_P)$ to get the toric bundle $\mc{E}_P$ over $\t{X}(\mS_P)$. Let $E$ be the fiber of the vector bundle $\mc{E}_P$ over $x_0$. By Theorem \ref{th:vb-dvr}, we have a piecewise affine linear function $h_P: | \mc{S}_P| \rar \t{\B}_P(E)$. This is affine linear on each polyhedron of $S_P$ by construction. Moreover, the linear part of $h_P$ corresponds to the toric vector bundle $\mc {E}_{\eta}$ over the toric variety $\t{X}(\mS_{\eta})$. This shows that condition (2) of Definition \ref{support} holds.

Now, let $\D \in \mS$ be such that the locus of $\D$ is affine. Since $\mc {E}|_{X(\D)}$ is a $T$-equivariant vector bundle, we will have an affine open invariant cover $\mc U$ of $X(\D)$ such that $\mc E$ is equivarianlty trivial over each open set in this cover (see \cite[Proposition 2.5]{IS}). By Lemma \ref{cover}, we may assume that there is an open cover $\mc U'$ of $Y$ such that \[\mc U= \{X(\D |_U) \mid U \in \mc U'\}.\]Note that, $\t{X}(\D |_U) = X(\D |_U)$, for $U \in \mc U'$. By Lemma \ref{trivial}, there is a basis $b_{\D,1}, \ldots, b_{\D,r}$ of $E$ and $u_{\D,1}, \ldots, u_{\D, r} \in M$ such that \[h_P(x)(\sum_{i=1}^r \lambda_i b_{\D,i})=  \min_i\{\val_P(\lambda_i)+\langle u_{\D, i}, x \rangle\}\] for all $P \in \Loc(\D|_U)$ and for all $x \in \D_P$. Here $\lambda_i \in K$ for $i=1, \ldots, r$. Note that $\Loc(\D|_U)= \Loc(\D) \cap U$ is a nonempty open subset of $Y$ whose complement in $Y$ is finite. This shows that our collection $(h_P)_{P\in Y}$ satisfies the condition (3) of Definition \ref{support}.

For $\D \in \mS$ be such that locus of $\D$ is complete, by Lemma \ref{cover} we know that $\mc {E}|_{X(\D)}$ is a $T$-equivariantly trivial vector bundle. Hence, its pullback $\mc {E}|_{\t{X}(\D)}$ is $T$-equivariantly trivial. Again appealing to Lemma \ref{trivial}, we see that our collection $(h_P)_{P\in Y}$ satisfies the condition (4) of Definition \ref{support}. 

Similarly, if we have a $T$-equivariant morphism of two vector bundles $\mc E\to \mc E'$, by restricting it to $x_0$, we get a map $E\to E'$ which by Theorem \ref{th:vb-dvr} induces a morphism of the corresponding support maps. 

For the converse, let $h=(h_P)_{P \in Y}$ be a support map. 
We would like to construct an equivariant vector bundle on $X(\mf D)$ with the corresponding support map $h$. The technical part of the construction is Lemma \ref{gluing}, which will be proved later. For $\D \in \mc S$ with affine locus, let $U_{\D}$ denote the open subset of $Y$ such that there is a basis $b_{\D,1}, \ldots, b_{\D,r}$ of $E$ and $u_{\D, 1}, \ldots, u_{\D, r} \in M$ such that \[h_P(x)(\sum_{i=1}^r \lambda_i b_{\D,i})=  \min_i\{\val_P(\lambda_i)+\langle u_{\D, i}, x \rangle\}\] for all $P \in U_{\D}$ and for all $x \in \D_P$. Here $\lambda_i \in K$ for $i=1, \ldots, r$. Consider the open subset \[U= \bigcap_{\{\D \in \mS \mid \Loc(\D) \text{ is affine}\}}U_{\D}\] of $Y$. Let $X'$ be the open subset of $X(\mS)$ given by \[X'=\left(\bigcup_{\{\D \in \mS \mid \Loc(\D) \text{ is affine}\}} X(\D |_U)\right) \bigcup \left(\bigcup_{\{\D \in \mS \mid \Loc\D \text{ is complete}\}} X(\D)\right).\] We first construct a $T$-equivariant vector bundle on $X'$. Note that if $\D$ has affine locus, then $\D |_U$ is a face of $\D$. Consider the set \[\mc{S}'= \{\D|_U \mid \D \in \mS \text{ such that} \Loc(\D) \text{ is affine}\} \cup \{\D \in \mS \mid \Loc(\D) \text{ is complete}\}.\] We observe that this is a divisorial fan. By Lemma \ref{gluing} below, we have a $T$-equivariant vector bundle $\mc{E}'$ on $X'$ such that for $\D \in \mS$ with affine locus, 
\begin{equation}\label{affine}
    H^0(X(\D|_U), \mc {E}') = \bigoplus_{i=1}^r H^0(X(\D |_U), \mc O)\chi^{u_{\D, i}} \ b_{\D, i}.
\end{equation}
and for $\D \in \mS$ with complete locus
\begin{equation}\label{complete}
    H^0(X(\D), \mc {E}') = \bigoplus_{i=1}^r H^0(X(\D), \mc O)\chi^{u_{\D, i}} \ b_{\D, i}. 
\end{equation}
Note that we have the surjection 
\begin{equation}\label{surjection}
X' \sqcup
    \bigsqcup_{\substack{P \notin U \\ \Loc(\D) \text{ is affine}}}\t{X}(\D_P) \longrightarrow X(\mS).
\end{equation}
Furthermore, if $\Loc(\D$) is affine, then $\t{X}(\D)=X(\D)$ and we have the following Cartesian diagram
\[
\begin{tikzcd}
    \widetilde X(\D)_P \ar[r] \ar[d] & \widetilde X(\D) \ar[d] \\
    \Spec \mc O_{Y,P} \ar[r]       & \Loc(\D)  .
\end{tikzcd}           
           \]
Since $\Spec \mc O_{Y,P} \rar \Loc(\D)$ is flat, so is $\widetilde X(\D)_P \rar \widetilde X(\D)=X(\D)$. Also, note that $\{\D|_U \mid \D \in \mS \text{ such that}\Loc(\D) \text{ is affine}\}$ is a finite set. Hence the surjection \eqref{surjection} is, in fact, an fpqc covering of $X(\mS)$. Observe that the intersection of $\t{X}(\D_P)$ with $\t{X}(\D_{P'})$ for another $P' \notin U$ or with $X'$ is an affine toric variety inside the toric variety $\t{X}(\mS_{\eta})$. Now, for each $P \notin U$ and $\D \in \mS$ such that $\Loc(\D)$ is affine, we have the $T$-equivariant vector bundle $\mc{E}_P$ on $X(\D)_P$ associated to the piecewise affine linear map $h_P$, and we have constructed a $T$-equivariant vector bundle $\mc{E}'$ on $X'$ above. These bundles agree on the intersection $\t{X}(\mS_{\eta})$ since the restriction corresponds to the linear part of $h_P$, which is the same for all $P$. Hence, by fpqc descent, we get a $T$-equivariant vector bundle over $X(\mS)$ (see \cite[\href{https://stacks.math.columbia.edu/tag/023T}{Proposition 023T}]{St}) whose restriction to $X'$ and $\t X(\mf D_P)$ are $\mc E'$ and $\mc E_P$ respectively.

Finally, suppose we have a morphism of support maps 
$$h=(h_P:| \mc S_P | \to \t{\B}_P(E))_{P\in Y}\to 
h'=(h'_P:| \mc S_P | \to \t{\B}_P(E'))_{P\in Y}\,.$$ 
We would like to show that it corresponds to a morphism of equivariant vector bundles. 
Let $\mc E$ and $\mc E'$ be the equivariant vector bundles corresponding to $h$ and $h'$, respectively.
By Lemma \ref{gluing of morphisms} below, we get a morphism $\mc E|_{X'}\to \mc E'|_{X'}$ and by Theorem \ref{th:vb-dvr}, we get morphisms $\mc E_P\to \mc E'_P$ over the toric schemes $X(\D_P)$, for $P\notin U$ and $\Loc(\D)$ affine.  These two agree on the intersections, and as a result, again, by descent for morphisms, we get the required morphism of vector bundles $\mc E \to \mc E'$. This completes the proof.
\end{proof}

\begin{rmk}  \label{rem-tvb-special-fiber}
For almost all $P \in Y$, the polyhedral complex $\mc S_P$ coincides with its tail fan. Then the fiber $\pi^{-1}(P)$ is a toric variety over $\bf k$ corresponding to the tail fan. One can recover the piecewise linear map $\Phi_P$ associated to the restriction of $\E$ to the fiber $\pi^{-1}(P)$ from the piecewise affine map $h_P$ as follows.
Let $\Lambda_P$ be the $\mc O$-lattice representing the lattice point $h_P(0) \in \t{\B}_P(E)$. Let $E_{\Lambda_P} = \Lambda_P \otimes_{\mc O_P} \bf k$ be the corresponding $\bf k$-vector space. Then $\Phi_P: |\tail(\mc S_P)| \to \t{\B}_\sph(E_{\Lambda_P})$ is the piecewise linear map obtained from $h_P: |\mc S_P| = |\tail(\mc S_P)| \to \t{\B}_P(E)$, by identifying a neighborhood of $\Lambda_P = h_P(0)$ in $\t{\B}_P(E)$ with a neighborhood of the extended Tits building $\t{\B}_\sph(E_{\Lambda_P})$ (see Theorem \ref{th-rest-irr-comp-sp-fiber}).
\end{rmk}

It remains to state and prove Lemmas \ref{gluing} and \ref{gluing of morphisms}.

\begin{lemma}[Gluing lemma] \label{gluing}
Let $\D^1$ and $\D^2$ be two pp divisors over $Y$ such that $\D^1 \cap \D^2$ is a face of both. Assume that for each $P \in \Loc( \D^1) \cup \Loc(\D^2) $, we have a piecewise affine linear map $h_P: \D^1_P \cup  \D^2_P \rar \t{\B}_P(E)$ which is affine linear on each $ \D^j_P $. Moreover, there is a basis $b_{\D^j,1}, \ldots, b_{\D^j,r}$ of $E$ and $u_{\D^j,1}, \ldots, u_{\D^j,r} \in M$ such that \[h_P(x)(\sum_{i=1}^r \lambda_i b_{\D^j,i})=  \min_i\{\val_P(\lambda_i)+\langle u_{\D^j,i}, x \rangle\}\] for all $P \in \Loc(\D^j)$ and for all $x \in \D^j_P$ $j=1, 2$. Here $\lambda_i \in K$ for $i=1, \ldots, r$. Then we have a $T$-equivariant vector bundle $\mc E$ over $X(\D^1) \cup X(\D^2)$ such that 
\begin{equation}\label{gluingeqn}
    H^0(X(\D^j), \mc {E}) = \bigoplus_{i=1}^r H^0(X(\D^j), \mc O)\chi^{u_{\D^j, i}} \ b_{\D^j, i}
\end{equation}
for $j=1,2$.
\end{lemma}
\begin{proof}
    For each $ j=1,2$, consider the trivial vector bundle $\mc {E}_{\D^j}$ on $X(\D^j)$ associated to the trivial module \[ H^0(X(\D^j), \mc {E}_{\D^j}) = \bigoplus_{i=1}^r H^0(X(\D^j), \mc O)\chi^{u_{\D^j, i}} \ b_{\D^j, i}.\] We show that the modules agree on the intersection $X(\D^1) \cap X(\D^2)=X(\D^1 \cap \D^2)$. It is enough to show that 
    \begin{equation}\label{component}
        H^0(X(\D^1 \cap \D^2), \mc {E}_{\D^1})_u= H^0(X(\D^1 \cap \D^2), \mc {E}_{\D^2})_u
    \end{equation}
    for all $u \in (\tail(\D^1 \cap \D^2))^{\vee} \cap M$. Now, we have 
    \begin{equation}\label{component2}
    \begin{split}
         H^0(X(\D^1 \cap \D^2), \mc {E}_{\D^j})_u&= \bigoplus_{\substack{i \text{ such that }\\ u-u_{\D^j, i}\in (\tail(\D^1 \cap \D^2))^{\vee} \cap M} } H^0(\Loc(\D^1 \cap \D^2), \mc {O}(\D^1 \cap \D^2)(u-u_{\D^j, i}))b_{\D^j,i}\\
         &= \bigoplus_{\substack{i \text{ such that }\\ u-u_{\D^j, i}\in (\tail(\D^1 \cap \D^2))^{\vee} \cap M} } \left[ \bigcap_{P \in \Loc(\D^1 \cap \D^2)}\mc{O}_{Y, P} \varpi_P^{- \lfloor \langle u-u_{\D^j, i},  \D^1_P \cap \D^2_P \rangle \rfloor} b_{\D^j,i}.\right]
    \end{split}
        \end{equation}
Here $\varpi_P$ is a uniformizing parameter of $\mc{O}_{Y, P}$. For $P \in \Loc(\D^1 \cap \D^2)$ and a fixed $i$, we have that 
\[\langle u-u_{\D^j, i},  \D^1_P \cap \D^2_P \rangle = \min_{\substack{v \text{ vertex of }\\ \D^1_P \cap \D^2_P}}\langle u-u_{\D^j, i},  v\rangle \]
This gives that for all $P \in \Loc(\D^1 \cap \D^2)$,
\begin{equation}\label{minimum}
    \begin{split}
        \mc{O}_{Y, P} \varpi_P^{- \lfloor\langle u-u_{\D^j, i},  \D^1_P \cap \D^2_P \rangle \rfloor} &=
        \bigcap_{\substack{v \text{ vertex of }\\ \D^1_P \cap \D^2_P}}\mc{O}_{Y, P} \varpi_P^{- \lfloor\langle u-u_{\D^j, i},  v \rangle \rfloor}\\
        &=\bigcap_{\substack{v \text{ vertex of }\\ \D^1_P \cap \D^2_P}}\mc{O}_{Y, P} \varpi_P^{ \lceil \langle u_{\D^j, i}-u,  v \rangle \rceil}.
    \end{split}
\end{equation}

Thus, by \eqref{component}, \eqref{component2} and \eqref{minimum}, it is enough to prove that for all $P \in \Loc(\D^1 \cap \D^2)$ and any vertex $v$ of $\D^1_P \cap \D^2_P$, the following holds 
\begin{equation}\label{transition}
   \bigoplus_{\substack{i \text{ such that }\\ u-u_{\D^1, i}\in (\tail(\D^1 \cap \D^2))^{\vee} \cap M} } \mc{O}_{Y, P} \varpi_P^{ \lceil \langle u_{\D^1, i}-u,  v \rangle \rceil}b_{\D^1,i}= \bigoplus_{\substack{i \text{ such that }\\ u-u_{\D^2, i}\in (\tail(\D^1 \cap \D^2))^{\vee} \cap M} } \mc{O}_{Y, P} \varpi_P^{ \lceil \langle u_{\D^2, i}-u,  v \rangle \rceil}b_{\D^2,i}.
\end{equation}

However, by Theorem \ref{th:vb-dvr}, for each $P \in \Loc(\D^1 \cap \D^2)$, we have a $T$-equivariant vector bundle $\mc{E}_P$ on the toric scheme $\U_{\D^1_P} \cup \U_{\D^2_P}$ associated to the map $h_P$. Moreover, we have 
\begin{equation}\label{toric-scheme-converse}
    H^0(\U_{\D^j_P}, \mc {E}_P) = \bigoplus_{i=1}^r H^0(\U_{\D^j_P}, \mc O)\chi^u_{\D^j,i} \ b_{\D^j, i}
\end{equation}

Fix a point $P \in \Loc(\D^1 \cap \D^2)$ and let $v$ be a vertex of $\D^1_P \cap \D^2_P$. Let $\rho_v$ be the ray generated by $v$ in $N_{\Q} \times \Q_{\geq 0}$. Thus the intersection of $\rho_v$ with $N_{\Q} \times \{1\}$ is $(v,1)$. Let $\Lambda_{\rho_v,u}$ denote the $u^{\text{th}}$-component of the module $ H^0(\U_{\rho_v}, \mc {E}_P)$. Since $\U_{\rho_v}$ is an open subset of both $\U_{\D^1_P}$ and $\U_{\D^2_P}$, by \eqref{affine-toric-scheme} and \eqref{toric-scheme-converse}, we have 
\begin{equation}\label{ray}
\begin{split}
    \Lambda_{\rho_v,u}&= \bigoplus_{i=1}^r \bigoplus_{\{k' \in \Q \mid  (u-u_{\D^1, i},k') \in \rho_v^{\vee} \cap \t{M}\}}\mc{O}_{Y, P} \varpi_P^{k'} \ b_{\D^1, i}\\
    &= \bigoplus_{i=1}^r \bigoplus_{\{k' \in \Q \mid  (u-u_{\D^2, i},k') \in \rho_v^{\vee} \cap \t{M}\}}\mc{O}_{Y, P} \varpi_P^{k'} \ b_{\D^2, i}.
\end{split}
    \end{equation}
Then following the arguments in \cite[Case-1, page 24]{KMT}, we have
\begin{equation*}
    \begin{split}
       (u-u_{\D^j, i},k') \in \rho_v^{\vee} & \Leftrightarrow 
       \langle u-u_{\D^j, i}, v \rangle + k' \geq 0 \\
       & \Leftrightarrow k' \geq \langle u_{\D^j, i}-u, v \rangle \\
       & \Leftrightarrow k' \geq \lceil \langle u_{\D^j, i}-u, v \rangle \rceil.
    \end{split}
\end{equation*}
This shows that
\begin{equation}\label{Case 1}
  \bigoplus_{i=1}^r \mc{O}_{Y, P} \varpi_P^{ \lceil \langle u_{\D^1, i}-u,  v \rangle \rceil}b_{\D^1,i}= \bigoplus_{i=1}^r \mc{O}_{Y, P} \varpi_P^{ \lceil \langle u_{\D^2, i}-u,  v \rangle \rceil}b_{\D^2,i}.
\end{equation}
Observe that $\D^j_P$ and $\D_j$ have the same tail cone by definition, and hence $\D_P^1 \cap \D_P^2$ and $\D^1 \cap \D^2$ have the same tail cone. Let $w$ be a primitive ray generator of $\tail(\D^1 \cap \D^2)$, considered in the vector space $N_{\Q}$. Let us denote by $\rho_w$, the ray generated by $w$ in $N_{\Q} \times \{0\}$. This is a ray of the cone over $\D_P^1 \cap \D_P^2$. Thus the intersection of $\rho_w$ with $N_{\Q} \times \{0\}$ is $(w,0)$. Now, following the arguments \cite[Case-2, page 24]{KMT}, we see that
\begin{equation*}
    (u-u_{\D^j, i},k') \in \rho_w^{\vee} \Leftrightarrow 
       \langle u-u_{\D^j, i}, w \rangle \geq 0
\end{equation*}
Since $K= \bigcup_{k' \in  \Z}\mc{O}_{Y, P} \varpi_P^{k'}$, by arguments similar to \eqref{ray}, we get:
\begin{equation}\label{Case 2}
   \Lambda_{\rho_w,u}= \bigoplus_{\substack{i \text{ such that }\\ \langle u-u_{\D^1, i}, w \rangle \geq 0}} K \ b_{\D^1,i}= \bigoplus_{\substack{i \text{ such that }\\ \langle u-u_{\D^2, i}, w \rangle \geq 0}} K \ b_{\D^2,i}.
\end{equation}
Taking intersection of \eqref{Case 2} over all the ray generators of $\tail(\D^1 \cap \D^2))$, we get 
\begin{equation}\label{tail}
\bigoplus_{\substack{i \text{ such that }\\  u-u_{\D^1, i}\in (\tail(\D^1 \cap \D^2))^{\vee} \cap M}} K \ b_{\D^1,i}= \bigoplus_{\substack{i \text{ such that }\\ u-u_{\D^2, i} \in (\tail(\D^1 \cap \D^2))^{\vee} \cap M}} K \ b_{\D^2,i}. 
\end{equation}
Thus, intersecting \eqref{Case 1} with \eqref{tail}, we obtain \eqref{transition}. This proves the Lemma.
\end{proof}

Now assume we have two vector spaces $E$ and $E'$ of rank $r$ and $r'$, respectively. 
Assume that we have two families of piecewise affine linear maps $$h_P:  \D^1_P \cup  \D^2_P \rar \t{\B}_P(E)$$ 
$$h'_P: \D^1_P \cup \D^2_P \rar \t{\B}_P(E')$$
for each $P \in \Loc(\D^1) \cup \Loc(\D^2) $ as in Lemma \ref{gluing}, i.e. there are bases $b_{\D^j,1}, \ldots, b_{\D^j,r}$ of $E$, $b'_{\D^j,1}, \ldots, b'_{\D^j,r'}$ of $E'$,
and $u_{\D^j,1}, \ldots, u_{\D^j,r} \in M$, $u'_{\D^j,1}, \ldots, u'_{\D^j,r'} \in M$ such that 
\[h_P(x)(\sum_{i=1}^r \lambda_i b_{\D^j,i})=  \min_i\{\val_P(\lambda_i)+\langle u_{\D^j,i}, x \rangle\}\] and
\[h'_P(x)(\sum_{i=1}^{r'} \lambda_i b'_{\D^j,i})=  \min_i\{\val_P(\lambda_i)+\langle u'_{\D^j,i}, x \rangle\}\]
for all $P \in \Loc(\D^j)$, for all $x \in \D^j_P$ $j=1, 2$. By Lemma \ref{gluing}, we have two vector bundles $\mc E$ and $ \mc E'$ on $X(\D^1) \cup X(\D^2)$. 
\begin{lemma}[Morphism gluing lemma] \label{gluing of morphisms}
     Suppose we have a morphism of $K$-vector spaces $\phi:E\to E'$ such that $\phi$ gives a morphism of piecewise affine maps $h_P\to h'_P$, $\forall~P\in {Loc}(\D^1) \cup {\Loc}(\D^2)$. Then we have a morphism of vector bundles $\mc E\to \mc E'$ whose fiber at $x_0$ is the map $\phi$.
\end{lemma}

\begin{proof}
    We need to show that under the induced map 
    $E\otimes \bk[M]\to E'\otimes \bk[M]$, the modules 
    $H^0(X(\D^j), \mc {E})$ map to $H^0(X(\D^j), \mc {E'})$ for $j=1,\,2$, and $H^0(X(\D^1 \cap \D^2), \mc {E})$ map to $H^0(X(\D^1\cap \D^2), \mc {E'})$. We check this only for $j=1$; the other cases are similar. 
    As before, it is enough to check that $H^0(X(\D^1), \mc {E})_u$ maps to $H^0(X(\D^1), \mc {E'})_u$ for any $u\in M$. Now, the same argument as in Lemma \ref{gluing} shows that 
     \begin{equation}
    \begin{split}
         H^0(X(\D^1), \mc {E})_u&= \Bigg[ \bigcap_{\substack{w \text{ primitive} \\ \text{generator of } \\ \tail(\D^1)^{\vee} \cap M}} \Lambda_{\rho_w,u} \Bigg] \bigcap \Bigg[ \bigcap_{P\in \Loc(\D^1)} \bigcap_{\substack{v \text{ vertex} \\ \text{of } \D_P}}  \Lambda_{\rho_v,u} \Bigg].
    \end{split}
        \end{equation}
        Here, $\rho_w$ is the ray generated by $(w,0)$ in $N_{\mb Q}\times \mb Q_{\geq 0}$, and for a vertex $v$ of $\D_P$, the ray $\rho_v$ is the ray generated by $(v,1)$. Recall that these rays together generate the cone $c(\widetilde{\D}_P)$. For a ray $\rho$ of $c(\widetilde{\D}_P)$, the module $\Lambda_{\rho,u}$ is $u$-th component of the $H^0(X(\rho), \mc E_P)$ where $\mc E_P$ is the vector bundle on $X(\D_P)$ corresponding to $h_P$. Similarly, we have that
        \begin{equation}
        \begin{split}
         H^0(X(\D^1), \mc {E'})_u&= \Bigg[ \bigcap_{\substack{w \text{ primitive} \\ \text{generator of } \\ \tail(\D^1)^{\vee} \cap M}} \Lambda'_{\rho_w,u} \Bigg] \bigcap \Bigg[ \bigcap_{P\in \Loc(\D^1)} \bigcap_{\substack{v \text{ vertex} \\ \text{of } \D_P}}  \Lambda'_{\rho_v,u} \Bigg]
    \end{split}
    \end{equation}
    Since $h_P\to h'_P$ is a morphism of piecewise affine linear maps for all $P\in \Loc(\D^1) \cup \Loc(\D^2)$, under the map $\phi$, the modules $\Lambda_{\rho_w,u}$ and $\Lambda_{\rho_v,u}$ map to $\Lambda'_{\rho_w,u}$ and $\Lambda'_{\rho_{v},u}$ respectively \cite[Proposition 2.9]{KMT}. From this, it follows that $H^0(X(\D^1), \mc {E})_u$ maps to $H^0(X(\D^1), \mc {E}')_u$. This completes the proof of this lemma.    
\end{proof}

\section{Applications}  \label{sec-applications}
In this section we present two applications of the support maps and the main theorem (\ref{thm-main1}), namely, a criterion for 
splitting of an equivariant vector bundle into a sum of line bundles, and a description of the space of global sections. 

\subsection{Criterion for equivariant splitting}  \label{sec-split}
We say that a torus equivariant vector bundle is \emph{equivariantly split} (or just \emph{split} for short) if it is equivariantly isomorphic to a sum of equivariant line bundles. 

\begin{theorem}  \label{th-splitting}
    Let $\E$ be a $T$-equivariant vector bundle over $X(\mS)$ and let $h$ be the associated support map. Then $\E$ splits equivariantly if and only if there exists a basis $B$ of $E$, the fiber of $\E$ over $x_0$, such that each $h_P$ factors through $\t{A}_P(B)$ for all $P \in Y$.
\end{theorem}
\begin{proof}
    Let $\E \cong \mc{L}_1\oplus \cdots \oplus \mc{L}_r$ where $\mc{L}_1, \ldots, \mc{L}_r$ are $T$-equivariant line bundles on $X(\mS)$ with associated support maps $h^1=(h^1_P)_{P \in Y}, \, \ldots, \, h^r=(h^r_P)_{P \in Y}$. Choose an identification of $E$ with $K^r$ via a basis $B=\{b_1, \ldots, b_r\}$. Then the support map associated to $\mc{L}_1\oplus \cdots \oplus \mc{L}_r$ is defined by $h=(h_P)_{P \in Y}$, where $h_P: | \mc{S}_P| \rar \t{\B}_P(E)$ is defined by \[h_P(x)(\sum_{i=1}^r\lambda_ib_i)= \min_i\{\val_P(\lambda_i)+ h_P^i(b_i)\},\] for all $x \in | \mc{S}_P|$. By Theorem \ref{thm-main1}, the forward direction follows.

    For the converse, let $B=\{b_1, \ldots, b_r\}$ be a basis of $E$ satisfying the condition in the statement. Fix $i \in \{1, \ldots, r\}$ and define 
    \begin{equation*}
        \begin{split}
            h^i_P:| \mc{S}_P| & \rar \t{A}_P(b_i)\\
            x & \mapsto h_P(x) |_{Kb_i}.
        \end{split}
    \end{equation*}
    Then $h^i=(h^i_P)_{P \in Y}$ defines an element of $\mc{C}(\mS)$. By Theorem \ref{thm-main1}, we  have the corresponding $T$-equivariant line bundle $\mc{L}_i$ on $X(\mS)$. Now, we see that $h$ coincides with the support map of $\mc{L}_1\oplus \cdots \oplus \mc{L}_r$, by the forward direction. Hence, by Theorem \ref{thm-main1}, we have $\E \cong \mc{L}_1\oplus \cdots \oplus \mc{L}_r$.
\end{proof}

\subsection{Space of global sections}  \label{sec-global-sec}
Let $\E$ be a $T$-equivariant vector bundle on a $T$-variety $X=X(\mS)$.  Then the natural $T$-action on the global sections gives the $T$-isotypical decomposition $H^0(X(\mS), \E) = \bigoplus_{u \in M} H^0(X, \E)_u$. We have the following combinatorial description of global sections of $\E$.

\begin{theorem}\label{thm-global}
Let $\E$ be a $T$-equivariant vector bundle on a complexity-one $T$-variety $X=X(\mS)$ with associated support map $h=(h_P: |\mS_P| \rar \t{\B}_P(E))_{P \in Y}$. Then 
    \begin{equation}\label{global}
\begin{split}
    H^0(X, \E)_u=\{e \in E \mid  \text{ for all } P \in Y, \, \langle u, v \rangle &\leq h_P(v)(e),v \text{ vertex of } \mS_P\\ \text{ and } \langle u, w \rangle &\leq \lin(h_P)(w)(e), w \in \tail(\mS_P)(1) \}. 
\end{split}
     \end{equation}
\end{theorem}

\begin{proof}
Any element of $H^0(X, \E)_u$ corresponds to a $T$-equivariant morphism of vector bundles $$\mc{O} (\chi^u) \rar \E.$$ Here $\mc{O} (\chi^u)$ denotes the trivial vector bundle over $X$ with the $T$-action given by the character $\chi^u$. Let $h^u=(h^u_P)_{P \in Y}$ be the support map associated to $\mc{O}( \chi^u)$, where $h^u_P: |\mS_P| \rar \t{\B}_P(K)$ is defined by $(h^u_P(v))(f)= \val_P(f) + \langle u, v \rangle$ for $v \in |\mS_P|, \, f \in K$. Now, by Theorem \ref{thm-main1}, a morphism $\mc{O}(\chi^u) \rar \E$ corresponds to a morphism of support maps $h^u \rar h$. Thus, we will have a map of $K$-vector spaces $\phi:K \rar E$ defined by $1 \mapsto e$ such that for all $P \in Y, \, v \in |\mS_P|, \, f \in K$, we have 
\begin{equation*}
    \begin{split}
        h^u_P(v)(f) &\leq h_P(v)(\phi(f))\\
        \text{or, } \val_P(f) + \langle u, v \rangle &\leq h_P(v)(fe)= \val_P(f) + h_P(v)(e)\\
        \text{or, } \langle u, v \rangle &\leq h_P(v)(e).
    \end{split}
\end{equation*}
Therefore, 
\begin{equation}
    H^0(X, \E)_u=\{e \in E \mid  \langle u, v \rangle \leq h_P(v)(e) \text{ for all } P \in Y, \, v \in |\mS_P|\}. 
\end{equation}
Observe that if for some $e \in E$, we have $\langle u, v \rangle \leq h_P(v)(e)$, for all $v \in |\mc S_P|$, then for any $w \in \tail(\mc S_P)$, we will have that $\langle u, w \rangle \leq \lin(h_P)(w)(e)$. To see this, let $w \in \sigma_P$ and $\D_P \in \mS_P$ be such that $\tail (\D_P)=\sigma_P$. Then we have characters $u_1, \ldots, u_r \in M$ and a basis $b_1, \ldots, b_r$ of $E$ such that for any $v \in \D_P$, we have $$h_P(v)(\sum_{i=1}^r\lambda_i b_i)=\min_i\{\val_P(\lambda_i) + \langle u_i, v \rangle\}.$$ Then, $$\lin(h_P)(w)(\sum_{i=1}^r\lambda_i b_i)= \min_i\{ \langle u_i, w \rangle\}.$$ Now, for above $e=\sum_{i=1}^r\lambda_i b_i$ and any $t \in \Q_{\geq 0}$, $$h_P(v+tw)(e) \geq \langle u, v+tw \rangle.$$ This implies that, for all $i$ such that $\lambda_i \neq 0$, 
\begin{equation*}
    \begin{split}
      \val_P(\lambda_i) + \langle u_i, v+tw \rangle &\geq \langle u, v+tw \rangle\\
      \Rightarrow \val_P(\lambda_i) + \langle u_i, v \rangle + t \langle u_i, w \rangle &\geq \langle u, v \rangle + t \langle u, w \rangle.  
    \end{split}
\end{equation*}
Now, dividing by $t$ and then letting $t$ tend to infinity, for all $i$ such that $\lambda_i \neq 0$, we have, $$\langle u_i, w \rangle \geq \langle u, w \rangle.$$ This means that $$\langle u, w \rangle \leq \lin(h_P)(w)(e).$$
By the above argument, any $e \in  H^0(X, \E)$ satisfies the description in the right-hand side of equation \eqref{global}.

Now, let $e \in E$ be such that for all $P \in Y,$
\begin{equation*}
    \begin{split}
       \langle u, v \rangle &\leq h_P(v)(e), \,v \text{ vertex of } \mS_P\\ 
       \text{ and }\langle u, w \rangle &\leq \lin(h_P)(w)(e), w \in \tail(\mS_P)(1). 
    \end{split}
\end{equation*}
Let $v \in \D_P$, then 
     $$v= \sum t_lv_l + \sum s_jw_j,$$
     for $v_l$ vertices of $\D_P,\, w_j$ rays of $\tail (\D_P)$, and $t_l \in \Q_{\geq 0}$ such that $\sum t_l=1$ and $s_j \geq 0.$
 As above, we have characters $u_1, \ldots, u_r \in M$ and a basis $b_1, \ldots, b_r$ of $E$ such that for any $v \in \D_P$, we have $$h_P(v)(\sum_{i=1}^r\lambda_i b_i)=\min_i\{\val_P(\lambda_i) + \langle u_i, v \rangle\} $$ and $$\lin h_P(w)(\sum_{i=1}^r\lambda_i b_i)= \min_i\{ \langle u_i, w \rangle\}.$$  Now for above $e=\sum_{i=1}^r\lambda_i b_i$ and for all $i$ such that $\lambda_i \neq 0$, $$\val_P(\lambda_i) + \langle u_i, v_l \rangle \geq \langle u, v_l \rangle \text{ and } \langle u_i, w_j \rangle \geq \langle u, w_j\rangle .$$ Thus for all $i$ such that $\lambda_i \neq 0$, 
\begin{equation*}
    \begin{split}
        \val_P(\lambda_i) + \langle u_i, \sum t_lv_l + \sum s_jw_j \rangle &=\val_P(\lambda_i) + \sum t_l \langle u_i, v_l \rangle + \sum s_j \langle u_i, w_j \rangle\\
        &\geq \sum t_l\langle u, v_l \rangle + \sum s_j \langle u, w_j \rangle\\
        &=\langle u, v \rangle.
    \end{split}
\end{equation*}
This completes the proof.
\end{proof}

Recall that the linear part of $h$ corresponds to the toric vector bundle $\E_{\eta}$ over the toric variety $\t{X}(\mS_{\eta})$ (see Theorem \ref{thm-main1}). Since $\mS_{\eta}= \tail(\mS)$, by \eqref{global}, we have 
\begin{equation*}
    \begin{split}
         H^0(\t{X}(\mS_{\eta}), \E_{\eta})_u=\{e \in E \mid  \text{ for all } P \in Y, \, \langle u, w \rangle &\leq \lin(h_P)(w)(e), w \in \tail(\mS_P)(1) \} 
    \end{split}
\end{equation*}
(cf. \cite[Corollary 4.1.3]{Klyachko}). Then \eqref{global} becomes
\begin{equation}\label{global2}
\begin{split}
    H^0(X, \E)_u=\{e \in H^0(\t{X}(\mS_{\eta}), \E_{\eta})_u \mid  \text{ for all } P \in Y, \, \langle u, v \rangle &\leq h_P(v)(e),v \text{ vertex of } \mS_P \}. 
\end{split}
     \end{equation}
Now, for $P \in Y$ and $v$ a vertex of $\mS_P$, define the additive norm on $H^0(\t{X}(\mS_{\eta}), \E_{\eta})_u$, given by
\begin{equation*}
    \begin{split}
       h_P^u(v) : H^0(\t{X}(\mS_{\eta}), \E_{\eta})_u & \rar K\\
       e &\mapsto \lfloor h_P(v)(e)- \langle u, v \rangle \rfloor.
    \end{split}
\end{equation*}
Note that shift and floor of additive norms are additive norms \cite[Proposition 2.7]{KMT}. Also, taking a minimum of additive norms gives an additive norm. Now for each $P \in Y$, define the additive norm 
\begin{equation*}
    \begin{split}
       g_P^u : H^0(\t{X}(\mS_{\eta}), \E_{\eta})_u & \rar K\\
       e &\mapsto \min_{v \text{ is a vertex of }\mS_P} h^u_P(v)(e).
    \end{split}
\end{equation*}
Then we have a support map $g^u=(g^u_P)_{P \in Y}$, associated to which we have a vector bundle $\E^u_Y$ on $Y$ by Theorem \ref{thm-main1}. Observe that, by Theorem \ref{thm-global} and \eqref{global2}, we have the following result generalizing \cite[Proposition 3.23]{PS}.
\begin{corollary} Let $\E$ be a $T$-equivariant vector bundle on a $T$-variety $X=X(\mS)$, and let $\E^u_Y$ be the associated vector bundle on $Y$ defined above. Then \[H^0(X, \E)_u=H^0(Y, \E^u_Y).\]
\end{corollary}

\section{Examples}  \label{sec-examples}
\subsection{Toric downgrades}
As usual $T$ is a torus with character lattice $M$ and cocharacter lattice $N$. 
Let $X_\Sigma$ be a $T$-toric variety corresponding to a complete fan $\Sigma$ in $N_\Q$. Let $T' \subset T$ be a codimension $1$ subtorus with lattice of one-parameter subgroups $N' \subset N$ such that $N/N' \cong \Z$ (i.e. $N/N'$ is torsion free). Then without loss of generality we can write $T = T' \times \bk^*$.

Let $\E$ be a $T$-toric vector bundle on $X_\Sigma$. Fix a point $x_0$ in the open $T$-orbit in $X_\Sigma$. Let $E_\bk = \E_{x_0}$ be the fiber over the distinguished point $x_0$. One knows that $\E$ corresponds to a piecewise linear map $\Phi: |\Sigma| \to \t{\B}(E_\bk)$.

We consider $X_\Sigma$ as a complexity-one $T'$-variety over the projective curve $Y = \mathbb{P}^1$. Let $E_K = K \otimes_\bk E_\bk$ where $K=\bk(Y)=\bk(t)$ is the field of rational polynomials in one variable. We would like to describe the support map $(h_P: |\mc S_P| \to \t{\B}_P(E_K))_{P \in \mathbb{P}^1}$ in terms of the piecewise linear map $\Phi$.

First we recall the divisorial fan data of $X_\Sigma$. 
Each cone $\sigma \in \Sigma$ gives rise to a polyhedral divisor $\mf D^\sigma$. 
Let $\pi: N_\Q \times \Q \to \Q$ be the projection on the second factor. We can divide the cones in $\Sigma$ into three groups: 
$$\Sigma^+ = \{ \sigma \in \Sigma \mid \sigma \subset N_\Q \times \Q_{\geq 0}\},$$
$$\Sigma^- = \{ \sigma \in \Sigma \mid \sigma \subset N_\Q \times \Q_{\leq 0}\},$$
$$\Sigma^0 = \Sigma \setminus (\Sigma^+ \cup \Sigma^-).$$
The polyhedral divisors $\mf D^\sigma$ with $\sigma$ in $\Sigma^+ \cup \Sigma^-$ have affine locus while $\mf D^\sigma$ with $\sigma \in \Sigma^0$ that do not lie in $\pi^{-1}(0)$ have complete locus. 

Recall that $\t{X}(\mc S)$ is the $T'$-variety with birational contraction $r: \t{X}(\mc S) \to X(\mc S)$ and good quotient $\t{X}(\mc S) \to Y$ (see Remark \ref{rem-tilde-X-S}). In our case, $\t{X}(\mc S) = \t{X}_\Sigma$ is the $T$-toric variety with a fan $\t{\Sigma}$ which is a refinement of $\Sigma$ such that the linear projection $\pi$ maps cones in $\t{\Sigma}$ to cones in the fan of $\mathbb{P}^1$. 

Let $P^+$, $P^- \in \mathbb{P}^1$ denote the $\bk^*$-fixed points corresponding to the rays $\Q_{\geq 0}$ and $\Q_{\leq 0}$ (in the fan of $\mathbb{P}^1$) respectively.
An infinitesimal neighborhood of $P^+$ (respectively $P^-$) in $\t{X}_\Sigma$ gives rise to a toric scheme $\t{X}^+$ (respectively $\t{X}^-$) over the DVR $\mathcal{O}_{P^+}$ (respectively $\mathcal{O}_{P^-}$). The polyhedral complex $\mc S_{P^+}$ (respectively, $\mc S_{P^-}$) is the intersection of $\Sigma$ with the hyperplane $N_\Q \times \{1\}$ (respectively the intersection of $\Sigma$ with the hyperplane $N_\Q \times \{-1\}$). The polyhedral complexes $\mc S_{P^+}$ and $\mc S_{P^-}$ are depicted in Figure \ref{fig:toric-downgrade} 

Let $P \in \mathbb{P}^1$. There is a natural map $$j_P: \t{\B}_\sph(E_\bk) \to \t{\B}_P(E_K).$$ The map $j_P$ sends the lattice point in $\t{\B}_\sph(E_\bk)$ determined by a decreasing $\Z$-filtration $(E_i)_{i \in \Z}$ in $E_\bk$ to the lattice point in $\t{\B}_P(E_K)$ determined by the $\mathcal{O}_{P}$-lattice $\Lambda$ in $E_K$ given by:
$$\Lambda = \sum_{i \in \Z} \mf{m}_P^i \otimes E_{-i},$$ where $\mf{m}_P$ denotes the maximal ideal of $\mc O_P$. 

One shows that the support map $(h_P: |\mc S_P| \to \t{\B}_P(E_K))_{P \in \mathbb{P}^1}$, corresponding to $\E$, regarded as a $T'$-equivariant vector bundle, is given by: 
$$h_P = j_P \circ \Phi_{|\,|\mc S_P|}: |\mc S_P| \to \t{\B}_\sph(E_\bk) \to \t{\B}_P(E_K).$$
Here, we regard $\mc S_P$ as lying in $N_\Q \times \{\pm 1\}$ or $N_\Q \times \{0\}$, depending on which $\bk^*$-orbit in $\mathbb{P}^1$ the point $P$ lies on. 

\begin{figure}
    \centering
    \includegraphics[height=8cm]{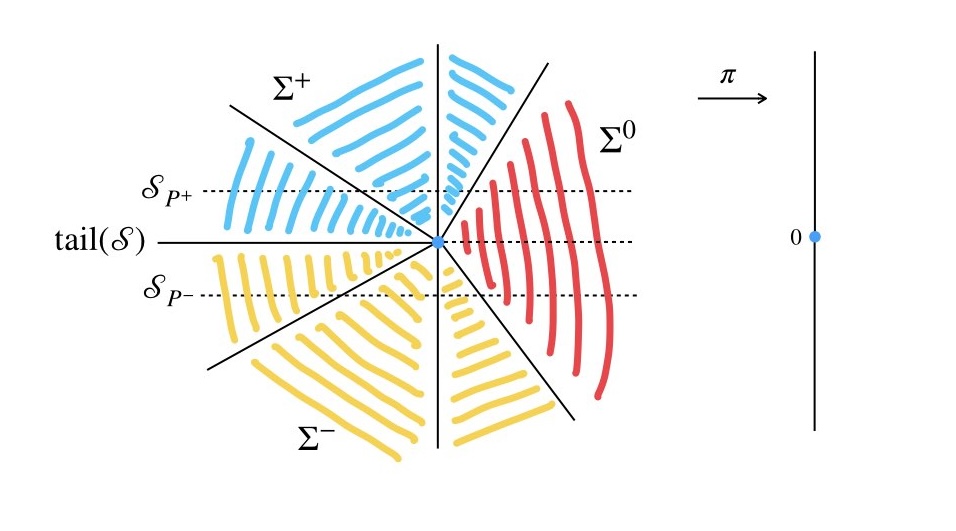}
    \caption{Fan of a toric surface as a complexity-one variety with the curve $Y = \mathbb{P}^1$}
    \label{fig:toric-downgrade}
\end{figure}

\subsection{(Co)tangent bundles}
Let $\mc S$ be a divisorial fan on $(Y,N)$, and let $X(\mc S)$ be the corresponding complexity-one $T$-variety over $\bk$.  Suppose that $X(\mc S)$ is smooth. We describe the data of the cotangent bundle $\Omega_{X(\mc S)/\bk}$. By Theorem \ref{thm-main1}, this reduces to describing the functions $\textup{cot}_P: |\mc S_P| \to \t{\B}_P(E_K)$, which classify the cotangent bundles $\Omega_{X(\mc S_P)/\bk}$ for $P \in Y$, where $E_K = (M\times \Z) \otimes K$. The upshot is that $\textup{cot}_P$ can be computed from the data of the cotangent bundle of the toric variety of the fan $c(\mc S_P) \subset N_\Q \times \Q_{\geq 0}$.  

Fix $p \in Y$. It suffices to describe this construction for each affine open $\U_\sigma = \Spec(R_\sigma)$, where $\sigma \in c(\mc S_P)$ (see Remark \ref{polyhedra and cone}). 
 Let $S_\sigma = \sigma^\vee \cap (M\times \Z)$, $V_\sigma = \Spec(\bk[S_\sigma])$, and $\t{V}_\sigma = \Spec(\mathcal{O}_P[S_\sigma])$. Suppose that $\Omega_{V_\sigma/\bk}$ has been given the following description as a module over $\bk[S_\sigma]$, where $u_i \in M$ and $b_i \in E_K$:
$$\Omega_{V_\sigma/\bk} = \bigoplus_{i =1}^{d+1} \left[\bigoplus_{(u - u_i, \ell) \in \sigma^\vee\cap (M\times \Z)} \chi^u t^\ell\bk b_i \right].$$
Then we must have:
$$\Omega_{\t{V}_\sigma/\mathcal{O}_P} = \bigoplus_{i =1}^{d+1} \left[\bigoplus_{(u - u_i, \ell) \in \sigma^\vee\cap (M\times \Z)} \chi^u t^\ell\mathcal{O} b_i \right].$$
From this we obtain a description of $\Omega_{\t{V}_\sigma/\bk}$:

$$\Omega_{\t{V}_\sigma/\bk} = \mathcal{O}_P[S_\sigma]d\varpi_P \oplus \bigoplus_{i =1}^{d+1} \left[\bigoplus_{(u - u_i, \ell) \in \sigma^\vee\cap M\times \Z} \chi^u t^\ell\mathcal{O} b_i \right].$$

Let $t \in \mathcal{O}_P[S_\sigma]$ denote the monomial corresponding to projection $N\times \Z \to \Z$. We have $R_\sigma \cong \mathcal{O}_P[S_\sigma]/\langle t - \varpi\rangle$, so we may use the exact sequence:
$$0 \to \langle t - \varpi\rangle/\langle t- \varpi\rangle^2 \to \Omega_{\t{V}_\sigma/\bk}\otimes_{\mathcal{O}_P[S_\sigma]} R_\sigma \to \Omega_{\U_\sigma/\bk} \to 0.$$
The quotient by the image of $t - \varpi$ identifies $d\varpi$ with an element in 
$$\bigoplus_{i =1}^{d+1} \left[\bigoplus_{(u - u_i, \ell) \in \sigma^\vee\cap M\times \Z} \chi^u t^\ell\mathcal{O} b_i \right] \otimes_{\mathcal{O}_P[S_\sigma]} R_\sigma.$$
Then, tensoring with $R_\sigma$ yields the expression:
$$\Omega_{\U_\sigma/\bk} = \bigoplus_{i =1}^{d+1} \left[\bigoplus_{(u - u_i, \ell) \in \sigma^\vee\cap M\times \Z} \chi^u \varpi^\ell\mathcal{O} b_i \right].$$
By \cite[Theorem 3.8]{KMT}, the corresponding function to $\t{\B}_P(E_K)$ has image in the apartment defined by the $b_i$, and can be computed by:
$$\Phi_\sigma(x)(\sum \lambda_i b_i) = \min\{\val_P(\lambda_i)+ \langle x, u_i \rangle\}.$$

\bibliographystyle{halpha}
\bibliography{Toric}
\end{document}